\documentclass[a4paper]{amsart}
\usepackage[leqno]{amsmath}
\usepackage{amssymb}
\usepackage{amscd}
\usepackage{amsthm}
\usepackage{mathrsfs}
\usepackage{mathtools}
\usepackage{stmaryrd}
\SetSymbolFont{stmry}{bold}{U}{stmry}{m}{n}
\usepackage{dsfont}
\usepackage{bbm}
\usepackage{bm}
\usepackage{xcolor}
\usepackage{accents}
\usepackage{hyperref}
\usepackage{enumerate}
\usepackage{cite}
\usepackage[utf8]{inputenc}
\usepackage[all,cmtip]{xy}
\usepackage{etoolbox}
\usepackage{tikz-cd}
\usetikzlibrary{arrows}
\usepackage{extarrows}

\numberwithin{equation}{section}

\newcommand{\Z}{\ensuremath{\mathbb{Z}}}
\newcommand{\Q}{\ensuremath{\mathbb{Q}}}
\newcommand{\R}{\ensuremath{\mathbb{R}}}
\newcommand{\C}{\ensuremath{\mathbb{C}}}

\DeclareMathOperator{\HH}{\mathrm{H}}
\DeclareMathOperator{\Hom}{Hom}
\DeclareMathOperator{\Span}{span}

\DeclareMathOperator{\Gal}{Gal}

\DeclareMathOperator{\Ker}{ker}

\DeclareMathOperator{\im}{im}

\DeclareMathOperator{\Tr}{Tr}
\DeclareMathOperator{\red}{red}
\DeclareMathOperator{\supp}{supp}
\DeclareMathOperator{\ord}{ord}
\DeclareMathOperator{\res}{res}
\DeclareMathOperator{\adj}{adj}

\DeclareMathOperator{\SL}{SL}
\DeclareMathOperator{\PSL}{PSL}
\DeclareMathOperator{\GL}{GL}
\DeclareMathOperator{\PGL}{PGL}
\DeclareMathOperator{\SO}{SO}
\DeclareMathOperator{\Ort}{O}
\DeclareMathOperator{\Spin}{Spin}
\DeclareMathOperator{\Mat}{M}
\newcommand{\PP}{\ensuremath{\mathbb{P}}}

\newcommand{\m}{\ensuremath{\mathfrak{m}}}

\newcommand{\cH}{\mathcal{H}}
\newcommand{\cT}{\mathcal{T}}
\newcommand{\cV}{\mathcal{V}}
\newcommand{\cE}{\mathcal{E}}
\newcommand{\cU}{\mathcal{U}}
\newcommand{\cF}{\mathcal{F}}
\newcommand{\cO}{\mathcal{O}}
\newcommand{\cL}{\mathcal{L}}

\newcommand{\cN}{\mathcal{N}}
\newcommand{\cP}{\mathcal{P}}
\newcommand{\cA}{\mathcal{A}}
\newcommand{\cM}{\mathcal{M}}
\newcommand{\cC}{\mathcal{C}}
\newcommand{\sD}{{\mathscr D}}

\newcommand{\cR}{\mathcal{R}}

\newcommand{\spe}{\mathrm{sp}}
\newcommand{\ratq}{\mathrm{rq}}

\newcommand{\orient}{\mathrm{or}}
\newcommand{\simplex}{\boldsymbol{\Delta}}
\DeclareMathOperator{\Div}{Div}
\DeclareMathOperator{\Coind}{Coind}
\DeclareMathOperator{\Ind}{Ind}
\DeclareMathOperator{\iso}{iso}

\renewcommand{\div}{\operatorname{div}}
\newcommand{\Orb}{\mathscr{O}}
\newcommand{\Gammao}{\Gamma_{\circ}}
\newcommand{\GroupD}{\mathcal{G}_{\Delta}}
\newcommand{\Group}{\mathcal{G}}

\newcommand{\into}{\hookrightarrow}
\newcommand{\onto}{\twoheadrightarrow}
\newcommand{\too}{\longrightarrow}								
\newcommand{\mapstoo}{\longmapsto}
\newcommand{\intoo}{\lhook\joinrel\longrightarrow}

\newtheorem{Lem}{Lemma}[section]
\makeatletter
\newlength{\@thlabel@width}%
\newcommand{\thmenumhspace}{\settowidth{\@thlabel@width}{\itshape1.}\sbox{\@labels}{\unhbox\@labels\hspace{\dimexpr-\leftmargin+\labelsep+\@thlabel@width-\itemindent}}}
\makeatother
\newtheorem{Pro}[Lem]{Proposition}
\newtheorem{Thm}[Lem]{Theorem}

\newtheorem{Cor}[Lem]{Corollary}
\theoremstyle{definition}

\newtheorem{Rem}[Lem]{Remark}

\newtheorem*{Que}{Question}

\author[L.~Gehrmann]{Lennart Gehrmann}
\address{L.~Gehrmann, Universität Bielefeld, Germany}
\email{gehrmann.math@gmail.com}
\author[S.~Sprehe]{Sören Sprehe}
\address{S.~Sprehe, Universität Bielefeld, Germany}
\email{ssprehe@math.uni-bielefeld.de}

\title[On the abundance of rigid meromorphic cocycles]{On the abundance of rigid meromorphic cocycles for quadratic forms in four variables}

\begin{document}

\begin{abstract}
Let $(V,q)$ be an anisotropic quadratic space of dimension $4$ over the rationals and let $p$ be a prime such that the local quadratic space $V\otimes \Q_p$ is the orthogonal direct sum of two hyperbolic planes.
We show that the group of $p$-adic rigid meromorphic cocycles attached to $V$ has infinite rank.
A key ingredient in the proof is the computation of the group of invertible rigid analytic functions on products of certain rigid analytic subspaces of the projective line such as the $p$-adic upper half-plane. 
\end{abstract}

\maketitle

\tableofcontents

\section*{Introduction}
\emph{Rigid meromorphic cocycles} for the group $\SL_2(\Q)$ were introduced by Darmon--Vonk in \cite{DarmonVonksingularmoduli} to define real quadratic analogues of the classical notion of singular moduli.
Roughly speaking, a rigid meromorphic cocycle is a class in the first cohomology of Ihara's group $\SL_2(\Z[1/p])$ with values in the multiplicative group of nonzero meromorphic functions on the $p$-adic upper half-plane $\cH_p=\PP^1(\C_p)\setminus \PP^1(\Q_p)$ with prescribed image under the divisor map.
More precisely, Darmon and Vonk associate to each $\SL_2(\Z[1/p])$-orbit of \emph{RM points} on $\cH_p$ a divisor-valued cohomology class, and define rigid meromorphic cocycles as those classes whose image under the divisor map are finite linear combinations of these classes. 
One of their main results of \textit{loc.cit.}~is that the group of rigid meromorphic cocycles has infinite rank by showing that:
\begin{enumerate}[(i)]
\item\label{Step1} the group of divisor-valued cohomology classes is large, that is, the cohomology classes attached to $\SL_2(\Z[1/p])$-orbits of RM points are linearly independent and
\item\label{Step2} the obstruction to lifting a divisor-valued cohomology class to a rigid meromorphic cocycle is small, e.g., it is torsion as a module of an appropriate Hecke algebra.
\end{enumerate}
Linear independence of the divisor-valued cohomology classes is an immediate consequence of the fact that they can be lifted to divisor-valued modular symbols.
The lifting obstruction can be computed via van der Put's description of the space of invertible analytic functions on $\cH_p$ in terms of harmonic cochains on the Bruhat--Tits tree.

In \cite{DGL}, a generalization of the theory of rigid meromorphic cocycles to orthogonal groups was proposed.
In that theory, the role of Ihara's group is replaced by a $p$-arithmetic subgroup $\Gamma$ of the orthogonal group $\Ort(V)$ of a nondegenerate rational quadratic space $(V,q)$ of dimension $\dim(V)\geq 3$ and real signature $(r,s)$, $r\geq s$, and the $p$-adic upper half-plane is replaced by a $p$-adic symmetric space $X_p$ of dimension $\dim(V)-2$, on which $\Ort(V)$ acts.
Similarly to the $\SL_2$-case, rigid meromorphic cocycles are defined as the classes in the degree-$s$ cohomology of $\Gamma$ with values in the group of nonzero meromorphic functions on $X_p$ with image under the divisor map given by explicit divisor-valued cohomology classes.
These divisor-valued classes are called \emph{Kudla--Millson divisors}.
In light of the results of Darmon--Vonk, it is natural to ask the following:
\begin{Que}
Does the group of rigid meromorphic cocycles for $\Gamma$ have infinite rank (under some mild conditions on $\Gamma$)?
\end{Que}
Naturally, the question breaks down into two parts:
\begin{enumerate}[(i)]
\item Is the space of Kudla--Millson divisors large?
\item Is the space of lifting obstructions small?
\end{enumerate}
For answering the first question, note that the approach of Darmon--Vonk is not applicable in general since the theory of modular symbols is only available if the quadratic space is not anisotropic.
But even if that is the case, not all Kudla--Millson divisors can be lifted to modular symbols:
let us assume for the moment that $V$ is not anisotropic and of signature $(r,1)$.
In that case, it is shown in \cite[Section 2.6]{DGL} that Kudla--Millson divisors lift to divisor-valued modular symbols if they satisfy an additional assumption.
This assumption is always satisfied in signature $(2,1)$, roughly half of the time in signature $(3,1)$ and never if $r \geq 5$.
We will show in Section \ref{sec: divcohomology} that the additional assumption is in fact necessary in signature $(3,1)$, if the local quadratic space $V_p\coloneq V\otimes \Q_p$ is the direct sum of two hyperbolic planes.

To tackle the question in case of a four-dimensional quadratic spaces, we follow the approach of \cite{GehrmannquaternionicRC}.
In \textit{loc.cit.}, the cohomology of a $p$-arithmetic subgroup $\Gamma$ of an inner form $G$ of $\SL_2$ with values in the group $\Div^\dagger(\cH_p)$ of locally finite divisors on $\cH_p$ is computed using the following two ingredients:
Firstly, there is a completely group theoretic description of $\Div^\dagger(\cH_p)$.
More precisely, $\Div^\dagger(\cH_p)$ can be written as a finite direct sum of $\Gamma$-modules of the form $\Coind_{\Gamma_\circ}^{\Gamma}(\oplus_{T}\Ind_{\Gamma_{\circ,T}}^{\Gamma_{\circ}}(\Z))$, where $\Gamma_\circ\subseteq \Gamma$ is an arithmetic subgroup, $T\subseteq F$ runs through a certain set of subtori and $\Gamma_{\circ,T}$ denotes the intersection $\Gamma_\circ \cap T$.
Secondly, Shapiro's lemma for homology and cohomology together with Bieri--Eckmann duality reduces the question to the computation of the integral homology of the groups $\Gamma_{\circ,T}$. 
The difficulty in generalizing the approach to the setting of orthogonal groups is that the group $\Div^\dagger_{\ratq}(X_p)$ of locally finite rational quadratic divisors on $X_p$ (see \cite[Section 2.3]{DGL}) in general does not have a simple description.
Nevertheless, in case the local quadratic space $V_p$ has Witt index $2$, that is, it is the direct sum of two hyperbolic planes, it has a two step resolution by modules with similar properties as in the $\SL_2$-case (see Proposition \ref{pro: divisorresolution}).
Analyzing this two-step filtration and applying Bieri--Eckmann duality arguments as in \cite{GehrmannquaternionicRC}
yields that $\HH^s(\Gamma, \Div^\dagger_{\ratq}(X_p))$ is a free $\Z$-module of infinite rank for any $p$-arithmetic subgroup $\Gamma\subseteq \SO(V)$ that is small enough provided that $V$ is anisotropic and the Witt index of $V_p$ is $2$ (see Theorem \ref{thm: divisors}).
Moreover, in Section \ref{sec: intersections} we prove that the explicit cohomology classes attached to $\Gamma$-orbits of positive lines in $V$ constructed in \cite[Section 2.4]{DGL} provide a basis of $\HH^s(\Gamma, \Div^\dagger_{\ratq}(X_p))$.

For answering the second question, we generalize van der Put's description of invertible analytic functions on analytic subsets of $\PP^1$.
Let us recall the classical result:
let $F$ be a complete non-Archimedean field and $\cL\subseteq \PP^1(F)$ be a compact subset.
The complement $\Omega_{\cL}$ of $\cL$ in $\PP^1_F$ carries the structure of a rigid analytic variety over $F$.
By a classical theorem of van der Put (see \cite[Theorem 2.7.11]{FvdP}) there is a natural isomorphism
\[
\cO(\Omega_{\cL})^\times\hspace{-0.2em}/F^\times \xlongrightarrow{\sim} M(\cL)_0,
\]
where $M(\cL)_0$ denotes the space of $\Z$-valued measures on $\cL$ of total mass $0$.
Let $\underline{\cL}=(\cL_1,\ldots, \cL_s)$ be a finite collection of compact subsets of $\PP^1(F)$ and consider the product $\Omega_{\underline{\cL}}\coloneq \prod_{i=1}^{s} \Omega_{\cL_i}$.
In Theorem \ref{thm: functions} below we show that the canonical homomorphism
\begin{align*}
\bigoplus_{i=1}^{s} \cO(\Omega_{\cL_i})^\times\hspace{-0.2em}/F^\times &\xlongrightarrow{\sim} \cO(\Omega_{\underline{\cL}})^\times\hspace{-0.2em}/F^\times\\
\quad (f_1,\ldots,f_s) &\mapstoo [(z_1,\ldots,z_s)\mapsto f_1(z_1)\cdot \ldots \cdot f_s(z_s)]
\end{align*}
is an isomorphism.
Returning to the setting of a four-dimensional rational quadratic space $V$, let $\cA \coloneq \cO(X_p)$ denote the ring of rigid analytic functions on the $2$-dimensional $\Q_p$-rigid analytic space $X_p$.
In case the local quadratic space $V_p$ is the sum of two hyperbolic planes, $X_p$ factors as a product of two $p$-adic upper half-planes.
Moreover, this factorization is compatible with the exceptional isomorphism between the Spin group of $V_p$ and the product of two copies of $\SL_2(\Q_p)$.
The isomorphism above, together with van der Put's result, yields an isomorphism
\[
\cA^\times\hspace{-0.2em}/\Q_p^\times \cong M(\PP^1(\Q_p))_0 \oplus M(\PP^1(\Q_p))_0
\]
compatible with the $\SL_2(\Q_p)\times \SL_2(\Q_p)$-actions on each side.
In case $V_p$ has Witt index $1$, we show that $X_p$ can be identified with a twisted version of a product of two $p$-adic upper half-planes associated to a quadratic extension $E$ of $\Q_p$.
Moreover, this identification is compatible with the exceptional isomorphism $\Spin(V_p)\cong \SL_2(E)$.
This yields a $\SL_2(E)$-equivariant isomorphism
\[
\cA^\times\hspace{-0.2em}/\Q_p^\times \cong M(\PP^1(E))_0.
\]
Additionally, we show that the Picard group of $X_p$ vanishes for both Witt index $2$ and Witt index $1$.
Using the standard resolution of the space of measures of total mass $0$ via the Bruhat--Tits tree, one can compute the cohomology of $p$-arithmetic subgroups of $\SO(V)$ with values in $\cA^\times$ in terms of spaces of automorphic forms.
For example, we show that $\HH^{s+1}(\Gamma,\cA^\times)$ is a finitely generated $\Z$-module if $V_p$ has Witt index $2$ (see Corollary \ref{cor: finiteness}) for every $p$-arithmetic subgroup $\Gamma\subseteq \SO(V)$. 
Moreover, we show that there are no interesting analytic theta cocycles, that is classes in $\HH^s(\Gamma,\cA^\times\hspace{-0.2em}/\Q_p^\times)$, if $V_p$ has Witt index $2$ and $s=1$.
More precisely, the rank of $\HH^1(\Gamma,\cA^\times\hspace{-0.2em}/\Q_p^\times)$ is $2$ in that case independent of the group $\Gamma$ (see Proposition \ref{pro: trivial}).

\medskip

\noindent \textbf{Acknowledgements}: Both authors received funding from Deutsche Forschungsgemeinschaft (DFG, German Research Foundation) via the grant SFB-TRR 358/1 2023 -- 491392403.

\medskip
\noindent \textbf{Notations}:
All rings are commutative.
Given a ring $R$, we write $R^\times$ for its unit group.
Let $G$ be a group.
A $G$-module is always a left $G$-module.
Given a right $G$-module, we also consider it as a $G$-module via precomposing with inversion on $G$. 
Let $H\subseteq G$ be a subgroup and $M$ an $H$-module.
Right multiplication turns the group algebra $\Z[G]$ into a $H$-module.
Recall that the coinduction of $M$ from $H$ to $G$ is the module
\[
\Coind_{H}^{G}(M)\coloneq \Hom_{\Z[H]}(\Z[G],M),
\]
on which $G$-action acts via left translation on $\Z[G]$.
More concretely, the coinduction can be described as the space of functions
\[
\Coind_{H}^{G}(M)\coloneq \{f\colon G \to M \mid f(gh)=h^{-1}(f(g))\ \forall h\in H, g\in G \}.
\]
The induction of $M$ from $H$ to $G$ is the defined as the tensor product
\[
\Ind_{H}^{G}(M)\coloneq \Z[G] \otimes_{\Z[H]} M
\]
with $G$-action given by left translation on $\Z[G]$.
Let $N$ be a $G$-module.
The map
\begin{align}\label{eq: indiso}
\begin{split}
\Hom_{\Z}(\Ind_{H}^{G}(M),N)&\xlongrightarrow{\sim}\Coind_{H}^{G}(\Hom_{\Z}(M,N))\\
\phi &\mapstoo [g \mapsto g^{-1}.\phi(g\otimes m)]
\end{split}
\end{align}
is an isomorphism of $G$-modules.


\section{Local geometry}\label{sec:geometry}
We fix a field $F$ which is complete with respect to a nontrivial non-Archimedean valuation $|\cdot|\colon F\to \mathbb{R}_{\geq 0}$.
Write $\cR_F$ for the valuation ring of $F$ with maximal ideal $\m$.
In the following, rigid analytic varieties will always be defined over $F$ unless stated otherwise. 
If $X$ is a rigid analytic variety, write $\cO$ for the sheaf of rigid analytic functions on $X$ and $\cO^\times$ for the sheaf of invertible rigid analytic functions on $X$.
If $f\in \cO(X)$ is a rigid analytic function on an affinoid space $X$, put $||f||_X\coloneq \sup_{x\in X}|f(x)|$.

\subsection{Invertible analytic functions} 
We consider the projective line $\PP^1_F$ viewed as a rigid analytic variety over $F$.
A point $z=[z_0:z_1]\in \PP^1_F$ is always assumed top be written in unimodular coordinates, that is, $|z_0|,|z_1|\leq 1$ and $|z_i|=1$ for at least one variable.
\subsubsection{Standard subsets.}
A \emph{standard open disc} is a proper subset of $\PP^1_F$ of the form $\{ [z_0:z_1]\in \PP^1_F \mid |\ell(z_0,z_1)|<r\}$ with $\ell\colon F^2 \to F$ a nonzero linear form and $r\in |F^\times|$.
Following \cite{vdP}, a rational subset of $\PP^1_F $ is called a \emph{standard subset} if it is the complement of a finite pairwise disjoint union of standard open discs.
If $\cC$ is a nonempty finite collection of nonempty pairwise disjoint standard open discs, then we denote by
\[
U_{\cC}\coloneq \PP^1_F\setminus \bigcup_{D\in\cC}D
\]
the corresponding standard subset.
Let us recall the description of invertible analytic functions on rigid analytic spaces of the form $U_{\cC}\times Y$ with $Y$ a connected affinoid space (see the lemma on page 195 of \cite{vdP}).

\begin{Lem}\label{lem: functions}
Let $U_{\cC}\subseteq \PP^1_F$ be a standard subset and $Y$ a geometrically connected affinoid space.
Every invertible rigid analytic function $f\in \cO(U_{\cC}\times Y)^\times$ can be written as a product
\begin{equation}\label{eq: functions}
f(x,y)=f_1(y)\cdot u(x,y)\cdot r(x)
\end{equation}
with $f_1\in \cO(Y)^\times$, $u\in \cO(U_{\cC}\times Y)^\times$, $||u-1||_{U_{\cC}\times Y} < 1$ and $r\colon \PP^1_F\to \PP^1_F$ a rational function with divisor supported in the union $\cup_{D\in\cC}D$.
Moreover, the integer
\[
\ord_{D}(f)\coloneq \sum_{x\in D} \ord_{x}(r) 
\]
for $D\in \cC$ does not depend on the decomposition.
\end{Lem}
\begin{proof}
To make the paper as self-contained as possible we will give a sketch of the proof.
First, let us assume that $Y$ consists of a point.
The existence of a decomposition as in \eqref{eq: functions} is a simple consequence of the Mittag-Leffler decomposition (see the proof of \cite[Proposition 2.5.10]{FvdP}).
It remains to show that $\ord_{D}(f)$ does not depend on the decomposition for every $D\in \cC$.
After applying an appropriate Möbius transformation, we may assume that $D=D(0,1)$ is the open disc around $0$ of radius $1$.
Let $B$ be the closed disc around $0$ of radius $1$.
Choose finitely many elements $a_1,\cdots,a_k \in \cR_F$ with $a_1=0$, $a_i \ncong a_j \bmod \m$ for $i\neq j$ such that $B'\coloneq B\setminus \cup_{i=1}^{k}D(a_i,1)\subseteq U_{\underline{D}}$, where $D(a_i,1)$ denotes the open disc around $a_i$ of radius $1$.
After rescaling we may assume that $||f||_{B'}=1$.
In particular, $f$ lies in the subring $\cO^\circ(B^\prime)\coloneq \{ g\in \cO(B^\prime) \mid  ||g||_{B^\prime}\leq 1\}\subseteq \cO(B')$.
The reduction of $f$ modulo the ideal $\cO^{\circ\circ}(B^\prime)\coloneq \{ g\in \cO^\circ(B^\prime) \mid  ||g||_{B^\prime}<1\}$ defines a rational function on the projective line over the residue field of $F$ and $\ord_{D}(f)$ is simply the order of vanishing at $0$ of that rational function (see the proof of \cite[Theorem 2.2.9]{FvdP}).

In the case of a general $Y$ we first note that the fact that $\ord_{D}(f)$ does not depend on the decomposition follows directly from the case that $Y$ is a point. For the existence, note that for any point $y\in Y$ the invertible rigid analytic function $f_y\coloneq f(x,y)\in \cO(U_{\cC})^\times$ admits a decomposition.
The assumption that $Y$ is geometrically connected implies that $\ord_D(f_y)$ does not depend on $y\in Y$ and, thus, one easily constructs the desired decomposition. For more details, see \cite[Lemma B.2.4]{SpreheThesis}.
\end{proof}

Given a profinite set $\cL$ we write $M(\cL,\Z)$ for the space of $\Z$-valued measures on $\cL$ and $M(\cL,\Z)_0\subseteq M(\cL,\Z)$ for the subspace of measures of total mass zero.
Write $\delta_x$ for the Dirac delta measure of $x\in \cL$. 
The last part of Lemma \ref{lem: functions} yields the homomorphism
\[
\ord_\cC\colon \cO(U_{\cC}\times Y)^\times\too M(\cC,\Z)_0,\quad f \mapstoo \sum_{D\in \cC} \ord_D(f) \delta_{D}.
\]
Let $U_{\cC'}$ be another standard subset containing $U_{\cC}$.
Then every $D' \in \cC'$ is contained in a unique $D\in \cC$.
The assignment $\delta_{D'}\mapsto \delta_{D}$ induces a homomorphism $M(\cC',\Z)_0\to M(\cC,\Z)_0$, which fits into the commutative diagram:
\begin{equation}\label{eq: ord}
\begin{tikzcd}
	{\cO(U_{\cC'}\times Y)^\times} & {} & M(\cC',\Z)_0 \\
	{\cO(U_{\cC} \times Y)^\times} && M(\cC,\Z)_0
	\arrow["\ord_{\cC'}", from=1-1, to=1-3]
	\arrow["\res",from=1-1, to=2-1]
	\arrow[from=1-3, to=2-3]
	\arrow["\ord_\cC", from=2-1, to=2-3]
\end{tikzcd}
\end{equation}

Given standard subsets $U_{\cC_1},\ldots , U_{\cC_s}\subseteq \PP^1_F$ indexed by the tuple $\underline{\cC}=(\cC_1,\ldots,\cC_s)$
we abbreviate 
\[
U_{\underline{\cC}}\coloneq U_{\cC_1} \times \cdots \times  U_{\cC_s}.
\]
Applying Lemma \ref{lem: functions} inductively one immediately gets a description of the space of invertible analytic functions on $U_{\underline{\cC}}$:
\begin{Cor}\label{cor: functions}
Let $U_{\cC_1},\ldots , U_{\cC_s}\subseteq \PP^1_F$ be standard subsets indexed by the tuple $\underline{\cC}=(\cC_1,\ldots,\cC_s)$.
Every invertible rigid analytic function $f\in \cO(U_{\underline{\cC}})^\times$ can be written as a product
\[
f(x_1,\ldots,x_s)=u(x_1,\ldots,x_s) \cdot\prod_{i=1}^{s} r_i(x_i)
\]
with $u\in \cO(U_{\underline{\cC}})^\times$ such that $||u-1||_{U_{\underline{\cC}}} < 1$ and $r_i\colon \PP^1_F\to \PP^1_F$ rational functions with divisor supported on $\cup_{D\in\cC_i}D$.
Moreover, the integer
\[
\ord_{i,D}(f)\coloneq \sum_{x\in D} \ord_{x}(r_i) 
\]
for $D\in \cC_i$ does not depend on the decomposition.
\end{Cor}
As in the case of a product of a standard set with a connected affinoid, the decomposition of Corollary \ref{cor: functions} yields a well-defined homomorphism
\[
\ord_{i}\colon U_{\underline{\cC}}\too M(\cC_i,Z)_0, \quad f  \mapstoo \sum_{D\in \cC_i} \ord_D(f) \delta_{D}.
\]
for every $1\leq i \leq s$.
We write $\ord_{\underline{\cC}}$ for the direct sum of the $\ord_{i}$, $1\leq i \leq s$.

Recall the following vanishing result of the higher cohomology groups of the sheaf of invertible analytic functions on generalized polydiscs by van der Put (see \cite[Corollary (3.10)]{vdP}).
\begin{Pro}\label{pro: vdP}
Let $X$ be a finite product of standard subsets of $\PP^1_{F}$. Then:
\[
\HH^i(X, \cO^\times)=0\quad \forall i >0.
\]
\end{Pro}

\subsubsection{Complements of compact sets.}
We utilize the above results to compute the cohomology of products of Drinfeld upper half spaces with values in $\cO^\times$.
More generally, we study products of spaces of the form
\[
\cH_{\cL}\coloneq \PP^1_F \setminus \cL
\]
with $\cL\subseteq \PP^1_F(F)$ compact.
Note that $\cH_{\cL}$ naturally carries the structure of a connected rigid analytic subspace of $\PP^1_F$.
More concretely, the compactness of $\cL$ implies that we may write $\cH_{\cL}$ as the union of an increasing sequence of standard subsets $U_{\cC(n)}$, $n\geq 1$.
Sending an element in $\cL$ to the unique disc of $\cC(n)$ containing it yields a homeomorphism
\[
\cL \xlongrightarrow{\sim}\varprojlim_n \cC(n).
\]
Thus, by the commutativity of the diagram \eqref{eq: ord} we get a homomorphism
\[
\cO(\cH_{\cL})^\times\hspace{-0.2em}/F^\times \too M(\cL,\Z)_0,
\]
which is well known to be an isomorphism (see \cite[Theorem 2.7.11]{FvdP}).
In particular, the only functions $f\in \cO(\cH_{\cL})^\times$ with $||f-1||_{\cH_{\cL}}<1$ are constant and the quotient $\cO(\cH_{\cL})^\times\hspace{-0.2em}/F^\times$ does not depend on the choice of base field $F$.
Moreover, $\HH^{i}(\cH_{\cL},\cO^\times)=0$ for all $i\geq 1$ (see \cite[2.7.6]{FvdP}).
In case $F$ is a local field, that is, $F$ is locally compact, we write $\cH_F\coloneq \cH_{\PP^1_F(F)}$ for Drinfeld's upper half-plane over $F$.
In that case, the isomorphism
\[
\cO(\cH_{F})^\times\hspace{-0.2em}/F^\times \too M(\PP^1(F),\Z)_0
\]
is equivariant with respect to the action of $\PGL_2(F)$ on both sides via Möbius transformations.

Given a tuple $\underline{\cL}=(\cL_1,\ldots,\cL_s)$ of compact subsets of $\PP^1(F)$ we put
\[
\cH_{\underline{\cL}}\coloneq \cH_{\cL_1}\times\cdots\times\cH_{\cL_s}.
\]
We may cover $\cH_{\underline{\cL}}$ by an increasing union of products of standard subsets $U_{\underline{\cC(n)}}$, $n\geq 1$.
As in the one-dimensional case, taking the limit of the homomorphisms $\ord_{\underline{\cC(n)}}$ yields a homomorphism
\begin{equation}\label{eq: prodmeasures}
\cO(\cH_{\underline{\cL}})^\times \too \bigoplus_{i=1}^{s} M(\cL_i,\Z)_0.
\end{equation}
\begin{Thm}\label{thm: functions}
Let $\underline{\cL}=(\cL_1,\ldots,\cL_s)$ be a tuple of compact nonempty subsets of $\PP^1_F(F)$.
The homomorphism \eqref{eq: prodmeasures} induces an isomorphism
\[
\cO(\cH_{\underline{\cL}})^\times\hspace{-0.2em}/F^\times \xlongrightarrow{\sim}  \bigoplus_{i=1}^{s} M(\cL_i,\Z)_0.
\]
Moreover:
\[
\HH^i(\cH_{\underline{\cL}},\cO^\times)=0 \quad \mbox{for all}\ i\geq 1.
\]
\end{Thm}
\begin{proof}
It follows from the one-dimensional case that \eqref{eq: prodmeasures} is surjective.
By Corollary \ref{cor: functions}, its kernel is given by functions $u$ that (up to multiplication with a constant) fulfil $||u-1||_{\cH_{\underline{\cL}}}< 1$.
From the one-dimensional case, one deduces that $u$ is constant.

Let $U_n\subseteq (\PP^1_F)^s$ be an increasing sequence of products of standard open subsets such that $\cH_{\underline{\cL}}=\bigcup_{n=1}^{\infty} U_{n}.$
By Proposition \ref{pro: vdP} it remains to show that the higher limit $\varprojlim^1_{n} \cO(U_n)^\times$ vanishes.
Let $(f_n)_n \in \prod_n \cO(U_n)^\times$ be an arbitrary sequence.
By Lemma \ref{cor: functions} we can write
\[
f_n=r_n\cdot u_n
\]
where $r_n$ extends to an invertible rigid analytic function on $\cH_{\underline{\cL}}$ and $||1-u_n||_{U_n}<1$.
Thus, the product $u'_n\coloneq\prod_{m>n}(u_m)$ converges to an invertible analytic function on $U_n$ and we define
\[
F_n\coloneq u_n'\cdot \prod_{i=1}^{n-1} r_{i}^{(-1)^{n-i}}.
\]
Then the sequence $(F_n)_n$ is a preimage of the sequence $(f_n)_n$ under the map
\[
\prod_{n=1}^{\infty}\cO(U_n)^\times\longrightarrow \prod_{n=1}^{\infty}\cO(U_n)^\times,\quad (g_n)_n\longmapsto (g_n/g_{n+1})_n,
\]
and, thus, the assertion follows.
\end{proof}

\begin{Rem}
The theorem above implies that every invertible rigid analytic function on $\cH_{\underline{\cL}}$ can be written as the product of functions on each factor and that this decomposition is unique up to multiplication with constants.
\end{Rem}

\subsection{Twisted diagonal divisors}\label{sec: divisors}

In this section, $F$ is assumed to be a local field.
Write $\ord_F\colon F^\times \onto \Z$ for the normalized valuation.
Let $\cT$ be the Bruhat--Tits tree of $G\coloneq \PGL_2(F)$ and $\red\colon\cH_F\to \cT$ the $G$-equivariant reduction map.
For a simplex $\sigma \in \cT$ put $\cU_\sigma\coloneq \red^{-1}(\sigma)\subseteq \cH_F$.
Given an element $A\in G$, we call 
\[
\Delta_A=\{(Az,z)\mid z\in\cH_F\}
\]
the \emph{(twisted) diagonal divisor} on $\cH_F^2$ attached to $A$.

The image of the diagonal divisor $\Delta_A$ under the reduction map is clearly given by the set
\[
\cT_A\coloneq\red(\Delta_A)=\{(A\sigma,\sigma)\mid \sigma\in \cT\}.
\]
\begin{Lem}
Let $A,B$ be elements of $G$.
Then $A=B$ if and only if $\cT_A=\cT_B$.
\end{Lem}
\begin{proof}
Suppose that $\cT_A=\cT_B$.
Then $(A\sigma,\sigma)=(B\sigma,\sigma)$ holds for all simplices $\sigma$ of $\cT$.
As the only element of $G$ that acts trivially on the Bruhat--Tits tree is the identity, the claim follows.
\end{proof}

\subsubsection{Locally finite divisors and locally finite functions.}
A \emph{locally finite diagonal divisor} on $\cH_F \times \cH_F$ is a formal (possibly infinite) linear combination
\[
\Delta=\sum_{A\in G} n_A(\Delta)\cdot \Delta_A,\quad n_A(\Delta)\in \Z,
\]
such that for every affinoid open $\cU\subseteq\cH_F^2$ the set
\begin{equation}\label{eq: supp}
\supp(\Delta\vert_{\cU})\coloneq\{A\in G\mid \Delta_A\cap \cU \neq \emptyset,\ n_A(\Delta)\neq 0 \}
\end{equation}
is finite.
It is enough to check finiteness of \eqref{eq: supp} for affinoids of the form $\cU_{(\sigma_1,\sigma_2)}\coloneq \cU_{\sigma_1}\times \cU_{\sigma_2}$ with $\sigma_1$ and $\sigma_2$ running through the set of simplices of $\cT$.
In fact, it is enough to consider the case that both, $\sigma_1$ and $\sigma_2$, are vertices.

Note that the intersection $\Delta_A\cap \cU_{(\sigma_1,\sigma_2)}$ is nonempty if and only if $(\sigma_1, \sigma_2)\in \cT_A$, which is equivalent to $A\sigma_2 =\sigma_1$.
Write $G_{\sigma}$ for the stabilizer of a simplex $\sigma$ of $\cT$ in $G$.
Given two simplices $\sigma_1,\sigma_2$ of $\cT$ of the same dimension, let $g_{\sigma_1,\sigma_2}\in G$ be any element with $g \sigma_1=\sigma_2$.
Then $\{A\in G\mid \Delta_A \cap \cU_{(\sigma_1,\sigma_2)} \neq \emptyset\}=g_{\sigma_1,\sigma_2} G_{\sigma_2}$ is a compact open subset of $G$.

Write $\Div_{\delta}^\dagger(\cH_F^2)$ for the group of locally finite diagonal divisors.
The action of $G^2$ on $\cH_F^2$ induces an action on $\Div_{\delta}^\dagger(\cH_F^2)$.
We may view a locally finite divisor $\Delta$ as a function on $G$ via the assignment
\[
 f_\Delta(A) \coloneq n_A(\Delta)\ \mbox{for}\ A\in G.
\]
The locally finiteness condition for $\Delta$ translates into a certain finiteness for the support of $f_\Delta$:
given a topological space $X$ and a commutative ring $R$ define
\[
\cF^\dagger(X,R)\coloneq \{f\colon X \to R \mid \#\supp(f\vert_U)< \infty\ \mbox{for all}\ U\subseteq X\ \mbox{compact}\}.
\]
Note that for any compact space $X$, $\cF^\dagger(X,R)$ is just the free $R$-module on $X$.
By the discussion above, we see that $f_\Delta$ is an element of $\cF^\dagger(G,\Z)$ and every $f\in \cF^\dagger(G,\Z)$ is equal to some $f_\Delta$.
Let $G^2$ act on $\cF^\dagger(G,\Z)$ via $((g_1,g_2).f)(g)=f(g_1^{-1}g g_2)$.
One immediately deduces the following:
\begin{Lem}
The map
\[
\Div_{\delta}^\dagger(\cH_F^2) \too \cF^\dagger(G,\Z),\quad \Delta \mapstoo f_\Delta,
\]
is a $\PGL_2(F)^2$-equivariant isomorphism.
\end{Lem}

Let $H$ be a locally profinite group, $h\in H$ an element, $K\subseteq H$ a compact open subgroup, and $R$ a commutative ring.
Put $K^h\coloneq h^{-1}Kh$.
Then, we have the following equality of function spaces
\[
\cF^\dagger(H,R)=\{f\colon H \to R \mid \#\supp(f\vert_{h'Kh})< \infty\ \mbox{for all}\ h'\in H\}.
\]
Restriction to $Kh$ gives a $K\times K^h$-equivariant surjection
\[
\cF^\dagger(H,R) \onto \cF^\dagger(Kh,R)
\]
which via Frobenius reciprocity yields a $H\times K^h$-equivariant homomorphism
\[
\cF^\dagger(H,R)\too \Coind_{K\times K^h}^{H\times K^h}(\cF^\dagger(Kh,R))
\]
that is easily seen to be an isomorphism.
\begin{Rem}
Since $K\times K^h$ acts transitively on $K$ and the stabilizer of $h\in K$ in $K\times K^h$ is given by the twisted diagonal $\delta(K)=\{(k,h^{-1}kh) \mid k\in K\}\subseteq K \times K$, there is a canonical isomorphism
\[
\cF^\dagger(Kh,R)\xlongrightarrow{\sim} \Ind_{\delta(K)}^{K\times K^h}(R)
\]
of $K\times K^h$-modules.
\end{Rem}
Applying the discussion above to the special case $H=G$ yields the following:
\begin{Lem}\label{lem: divisoriso}
Let $K\subseteq G$ be a compact open subgroup and $g\in G$.
There is a canonical isomorphism
\[
\Div_{\delta}^\dagger(\cH_F^2) \xlongrightarrow{\sim} \Coind_{K\times K^g}^{G\times K^g} (\Z[Kg])
\]
of $G\times K^g$-modules.
\end{Lem}

\subsubsection{Resolution via the Bruhat--Tits tree.}
In order to remedy the fact that we have to restrict to a compact open subgroup in the second variable, we replace $\Div_{\delta}^\dagger(\cH_F^2)$ by a two-step resolution coming from the Bruhat--Tits tree $\cT$ (cf.~\cite[Section II.2.8]{SerreTrees}).
Let $\cV$ be the set of vertices of $\cT$ and $\cE^{\orient}$ the set of oriented edges of $\cT$.
Given an oriented edge $e\in \cE^{\orient}$, write $s(e)$ and $t(e)$ for its source and target, respectively, and denote by $\bar{e}$ the edge with opposite orientation.
Consider the $G$-modules
\[
C_0^{\orient}(\cT)\coloneq \Z[\cV]\quad\mbox{and}\quad C_1^{\orient}(\cT)\coloneq \Z[\cE^{\orient}]/\Z[e+\bar{e}\mid e\in \cE^{or}]
\]
of oriented chains on $\cT$.
Since $\cT$ is a tree, the two homomorphisms
\begin{align*}
\partial\colon C_1^{\orient}(\cT)\too C_0^{\orient}(\cT),&\quad e \mapstoo t(e) - s(e),\\
\intertext{and}
\epsilon\colon C_0^{\orient}(\cT)\mapstoo \Z,&\quad \ v \mapstoo 1,
\end{align*}
give rise to the exact sequence
\begin{equation}\label{eq: treehomology}
0 \too C_1^{\orient}(\cT)\xlongrightarrow{\partial} C_0^{\orient}(\cT)\xlongrightarrow{\epsilon}\Z\too 0
\end{equation}
of $G$-modules.
Note that $G$ acts transitively on the sets $\cV$ and $\cE^{\orient}$.
Let $e$ be an oriented edge of $\cT$ with origin $v$.
As before, write $G_v$ for the stabilizer of $v$ and $G_e$ for the stabilizer of the underlying nonoriented edge.
Consider the character
\[
\chi_{\orient}\colon G \too \Z/2\Z,\quad g \mapstoo \ord_F(\det(g)).
\]
The assignments
\begin{align*}
\Ind_{G_v}^{G}(\Z)\xlongrightarrow{\sim} C_0^{\orient}(\cT),&\quad \sum_{g\in G} g\otimes n_g  \mapstoo \sum_{g\in G} n_g \cdot g.v,\\
\intertext{and}
\Ind_{G_e}^{G}(\chi_{\orient})\xlongrightarrow{\sim} C_1^{\orient}(\cT),&\quad \sum_{g\in G} g\otimes n_g  \mapstoo \sum_{g\in G} n_g \cdot g.e,
\end{align*}
define $G$-equivariant isomorphisms.
Applying the functor $\Hom_\Z(-,M)$ for a $G$-module $M$ to the short exact sequence \eqref{eq: treehomology} yields the short exact sequence
\[
0 \too M \xlongrightarrow{\epsilon^\ast} \Hom_\Z(C_0^{\orient}(\cT), M) \xlongrightarrow{\partial^\ast}\Hom_\Z(C_1^{\orient}(\cT), M)\too 0
\]
of $G$-modules.
Using the isomorphisms between spaces of oriented chains on $\cT$ and induced representations together with the identification \eqref{eq: indiso} the short exact sequence above becomes the short exact sequence
\begin{equation}\label{eq: BTcomplex}
0 \too M \too \Coind_{G_v}^{G} (M) \xlongrightarrow{\partial^\ast}\Coind_{G_e}^{G} (\chi_{\orient}\otimes M)\too 0.
\end{equation}
Consider the special case $M=\Div_{\delta}^\dagger(\cH_F^2)$ where $G$ acts via Möbius transformations on the second variable.
Since the action of $G$ on the first variable commutes with this $G$-action, the short exact sequence \eqref{eq: BTcomplex} is equivariant with respect to the $G\times G$-action.
Combining the discussion above with Lemma \ref{lem: divisoriso} implies the following:
\begin{Pro}\label{pro: divisorresolution}
The sequence
\[
0 \too \Div_{\delta}^\dagger(\cH_F^2) \too \Coind_{G_v^2}^{G^2}\left(\Z[G_v]\right)
\xlongrightarrow{\partial^\ast}
\Coind_{G_e^2}^{G^2} \left(\Z[G_e]\otimes \chi_{\orient}\right) \too 0
\]
of $G\times G$-modules is exact.
\end{Pro}

\subsubsection{Parity decomposition.}\label{sec: SL2}
Let us give a variant of the above resolution, in which we replace $G$ by the subgroup $G_+\coloneq\ker(\chi_{\orient})$ of even elements.
Its complement $G_-\coloneq G\setminus G_+$ is the set of odd elements.
We say that two vertices $v_1, v_2$ of $\cT$ have the same parity if their distance in $\cT$ is even.
Parity of vertices is preserved by the action of $G_+$.
Moreover, $G^+$ acts transitively on the set of vertices having the same parity.
If $A$ is even, then $v_1$ and $v_2$ have the same parity for every $(v_1,v_2)\in\cT_A$.
If $A$ is odd, $v_1$ and $v_2$ have the opposite parity for every $(v_1,v_2)\in\cT_A$.
The decomposition
\[
\Div^\dagger_\delta(\cH_F^2)=\Div^\dagger_\delta(\cH_F^2)^{+} \oplus \Div^\dagger_\delta(\cH_F^2)^{-}
\]
where
\[
\Div^\dagger_\delta(\cH_F^2)^{\pm}\coloneq \{\Delta \in \Div^\dagger_\delta(\cH_F^2) \mid n_A(\Delta)= 0\ \forall A\in G_{\mp}\}
\]
is stable for the subgroup of elements $(g_1,g_2)\in G \times G$ such that $\chi_{\orient}(g_1)=\chi_{\orient}(g_2)$.
We may decompose the resolution of Proposition \ref{pro: divisorresolution} accordingly:
let $v_0,v_1$ be two adjacent vertices of $\cT$ connected by the nonoriented edge $e$.
Write $K_0$, $K_1$ for the stabilizers of $v_0$ and $v_1$ in $G_+$, respectively, and $I_e$ for the stabilizer of $e$ in $G_+$.
Pick $g\in G_e$ with $\chi_{\orient}(g)=-1$.
Then there are $G_+\times G_+$-equivariant exact sequences
\begin{align*}
0 \too \Div^\dagger_\delta(\cH_F^2)^{+}\too \bigoplus_{i=0}^{1} \Coind_{K_i^2}^{G_+^2} \left(\Z[K_i]\right) \too  \Coind_{I^2}^{G_+^2} \left(\Z[I]\right) \too 0 
\end{align*}
and
\begin{align*}
0 \too \Div^\dagger_\delta(\cH_F^2)^{-}\too \bigoplus_{i=0}^{1} \Coind_{K_i\times K_{1-i}}^{G_+^2} \left(\Z[K_i g]\right) \too  \Coind_{I^2}^{G_+^2} \left(\Z[Ig]\right) \too 0 .
\end{align*}
The subgroup $\PSL_2(F)\subseteq G^+$ also acts transitively on the set of vertices of given parity.
Thus, one may replace $G_+$ by $\PSL_2(F)$ in the discussion above.

\subsection{Symmetric spaces}\label{sec: spaces}
We recall the definition of non-Archimedean symmetric spaces attached to orthogonal groups of Darmon--Gehrmann--Lipnowski.
In \cite{DGL}, the base field is assumed to be $\Q_p$, $p\neq 2$, and the quadratic space over $\Q_p$ is assumed to have a self-dual $\Z_p$-lattice.
We allow more general base fields but all the basic statements made in \textit{loc.cit.}~hold verbatim in this greater generality.
More precisely, we assume that $F$ is a local field of characteristic $\mathrm{char}(F)\neq 2$ and $(V,q)$ is a nondegenerate finite-dimensional quadratic space over $F$ of dimension $n\geq 3$ that is not anisotropic.
Let $\langle\cdot,\cdot\rangle$ be the bilinear form attached to $q$, that is, $2q(v)=\langle v,v\rangle$.
Given a vector $v\in V$ we write $f_v$ for the linear form on $V$ given by $f_v(w)=\langle v,w\rangle$.
Consider the smooth quadric $Q_V\subseteq \PP_V$ cut out by the equation $q=0$.
For a line $\ell=[v]\in \PP_V(F)$ let $H_v\subseteq \PP_V$ be the hyperplane cut out by $f_v=0$.
Our assumption on $V$ implies that $Q_V(F)\neq \emptyset$.
As explained in \cite[Section 1.3]{DGL},
\[
X_V\coloneq Q_V \setminus \bigcup_{\ell\in Q_V(F)} H_\ell
\]
naturally carries the structure of a connected rigid analytic subspace of $Q_V$.
Moreover, $Q_V$ is a Stein space.
The orthogonal group $\Ort(V)$ of $V$ acts on $X_V$.
Note that multiplying the quadratic form with a nonzero scalar does not change the space $X_V$.

For $\ell\in \PP_V(F)\setminus Q_V(F)$ we consider the divisor $\Delta_{\ell}=H_\ell\cap X_V$ on $X_V$. 
If $n\neq 4$, $\Delta_{\ell}$ is a prime divisor and $\Delta_\ell =\Delta_{\ell'}$ if and only if $\ell=\ell'$.
In case $n=3$, $\Delta_{\ell}$ is nonempty if and only if the orthogonal complement of $\ell$ is not a hyperbolic plane.
In that case $\Delta_{\ell}$ consists of two points defined over a quadratic extension of $F$.
Moreover, $\Delta_{\ell}\cap \Delta_{\ell'}=\emptyset$ if $\ell\neq \ell'$.
Put $\cN_V\coloneq \PP_V(F)\setminus Q_V(F)$ in case $n \geq 4$ and, if $n=3$, let $\cN_V\subseteq \PP_V(F)\setminus Q_V(F)$ be the subset of lines whose orthogonal complement is not a hyperbolic plane.
By definition, $\cN_V$ is stable under the $\Ort(V)$-action on $\PP_V(F)$.

A \emph{locally finite special divisor} on $X_V$ is a formal (possibly infinite) linear combination
\[
\Delta=\sum_{\ell\in \cN_V} n_\ell(\Delta)\cdot \Delta_\ell,\quad n_\ell(\Delta)\in \Z,
\]
such that for every affinoid open subset $\cU\subseteq X_V$ the set
\[
\supp(\Delta\vert_{\cU})\coloneq\{\ell\in \cN_V\mid \Delta_\ell\cap \cU \neq \emptyset,\ n_\ell(\Delta)\neq 0 \}
\]
is finite.
Write $\Div_{\spe}^\dagger(X_V)$ for the group of locally finite special divisors on $X_V$.
The action of $\Ort(V)$ on $X_V$ induces an action on $\Div_{\spe}^\dagger(X_V)$.
\begin{Rem}\label{rem: isotropy}
Similarly, as in the case of twisted diagonal divisors we can view a locally finite special divisor as a function from $\cN_V$ to $\Z$.
The calculations in \cite[Section 2.3]{DGL} show that this defines a $\Ort(V)$-equivariant isomorphism between $\Div_{\spe}^\dagger(X_V)$ and $\cF^\dagger(\cN_V,\Z)$, where $\Ort(V)$ acts on $\cF^\dagger(\cN_V,\Z)$ via left translation under the assumption that $V$ admits a self-dual $\cR_F$-lattice.
More precisely, let $\Lambda\subseteq V$ be any $\cR_F$-lattice.
For $\ell\in \cN_V$ fix a generator $v_\ell$ that is a primitive vector of $\Lambda$ and put
\[
\iso_\Lambda(\ell)\coloneq|(q(v_\ell)|\in |F^\times|.
\] 
The sets $\cN_{\Lambda,r}\coloneq \{\ell \in \cN_V\ \vert \iso_{\Lambda}(\ell)=r\}$ for $r\in |F^\times|$ form a cover of $\cN_V$ by compact open sets.
In \textit{loc.cit.} it is shown that, if $\Lambda$ is self-dual, then a formal special divisor $\Delta$ is locally finite if and only if the set $\{\ell \in \cN_{\Lambda,r}\mid n_\ell(\Delta) \neq 0\}$ is finite for every $r\in |F^\times|$.
In the following, we will show that $\Div_{\spe}^\dagger(X_V)$ and $\cF^\dagger(\cN_V,\Z)$ are isomorphic for all nonanisotropic four-dimensional quadratic spaces without the assumption on the existence of a self-dual lattice by considering explicit models for the quadratic space $V$.
\end{Rem}

We will give more concrete descriptions of these notions in the case of four-dimensional quadratic spaces and link them to the products of Drinfeld upper half-planes studied in the previous sections.
In particular, we will show that
\[
\HH^1(X_V,\cO^\times)=0
\]
for every four-dimensional quadratic space $V$ that is not anisotropic.

\subsubsection{The split case.}
Let us first consider the case that $V$ is the direct sum of two hyperbolic planes or, in other words, $V=\Mat_2(F)$ is the space of $2 \times 2$-matrices over $F$ with $q=\det$.
The corresponding bilinear form is given by $\langle M, N \rangle = \Tr(M \adj(N))$ where $\adj(N)$ denotes the adjunct matrix of $N$.
The orthogonal group of $V$ can be described explicitly in this case:
the group of matrices $\{(A_1,A_2\in \GL_2(F)^2 \ \vert \det(A_1)=\det(A_2)\}$ acts on $V$ via $(A_1,A_2)M=A_1 M A_2^{-1}$.
This action defines an element of $\SO(V)$ and every orthogonal transformation of $V$ of determinant one is of this form.
The subgroup of pairs of matrices acting trivially on $V$ is given by $Z=\{(\lambda,\lambda)\mid \lambda\in F^\times\}$.
Finally, the homomorphism sending a matrix to its adjunct matrix defines an element of $\Ort(V)\setminus \SO(V)$.

The quadric $Q_V$ is given by $2\times 2$-matrices of rank $1$ up to scaling.
Sending the line $[M]$ spanned by such a matrix to its kernel and image yields an isomorphism
\begin{equation}\label{eq: isom Q}
Q_V \xlongrightarrow{\sim} \PP^1_F \times \PP^1_F,\quad [M] \mapstoo (\im(M),\Ker(M)).
\end{equation}
The isomorphism is compatible with the $\SO(V)$-action, where a pair of matrices $(A_1,A_2)$ acts on $\PP^1_F\times \PP^1_F$ via $(A_1,A_2)(\ell_1,\ell_2)=(A_1 \ell_1,A_2 \ell_2)$.
Moreover, sending a matrix to its adjunct induces switching the two factors of $\PP^1_F\times \PP^1_F$.
Let $\ell$ be an element of $Q_V$ corresponding to the point $(x,y)\in \PP^1(F)\times \PP^1(F)$.
Then the image of the intersection $H_\ell\cap Q_V$ under \eqref{eq: isom Q} is equal to the closed subset $\{(x',y')\in \PP^1_F \times \PP^1_F\mid x=x'\mbox{ or } y=y'\}$.
Indeed, if $[M]\in H_\ell\cap Q_V$, then $M+v$ is isotropic.
We may assume that $M+v \neq 0$.
Then, both $\ker(M+v)$ and $\im(M+v)$ are $1$-dimensional.
Therefore, if $\ker(M)\neq \ker(v)$, then $\im(M),\im(v)\subset \im(M+v)$, which implies that $\im(M)=\im(v)$.
Similarly, if $\im(M)\neq \im(v)$, then $\ker(M)=\ker(v)$.
Thus, \eqref{eq: isom Q} induces an isomorphism
\begin{equation}\label{eq: isom H}
X_V\xlongrightarrow{\sim} \cH_F \times \cH_F.
\end{equation}
Theorem \ref{thm: functions} combined with \eqref{eq: isom H} yields an $\Ort(V)$-equivariant isomorphism
\[
\cO(X_V)^\times\hspace{-0.2em}/F^\times\cong M(\PP^1(F),\Z)_0 \oplus M(\PP^1(F),\Z)_0.
\]
Moreover, $\HH^i(X_V, \cO^\times)=0$ for all $i\geq 1$.
Under the isomorphism \eqref{eq: isom H}, the divisor $\Delta_A$ for $A\in \cN_V=\PGL_2(F)$ corresponds to the twisted diagonal divisor attached to $A$.
In particular, it induces an $\Ort(V)$-equivariant isomorphism
\[
\Div_{\spe}^\dagger(X_V) \xlongrightarrow{\sim}\Div_{\delta}^\dagger(\cH_F \times \cH_F).
\]

\subsubsection{The nonsplit case.}
Now let us assume that $V$ is not the direct sum of two hyperbolic planes.
Recall that we always assume that $V$ is not anisotropic.
Thus, $V$ is isomorphic (up to multiplying the quadratic form with a nonzero scalar) to the direct sum of a hyperbolic plane and the norm form of a quadratic extension $E/F$ that is uniquely determined by $V$.
In other words, $V$ is isomorphic to the space of matrices $\{M \in \Mat_2(E) \mid M' =-\adj(M)\}$, where $M'$ denotes the Galois conjugate of $M$ over $F$, with the quadratic form given by the determinant.
The group $\{A\in \GL_2(E)\mid \det(A)\in F^\times\}$ acts on $V$ via $A.M=A M (A')^{-1}$.
This action defines an element of $\SO(V)$ and every orthogonal transformation of $V$ of determinant one is of this form.
The subgroup of matrices acting trivially on $V$ is given by the scalar matrices corresponding to elements in $F^\times$.
The canonical map $V_E\coloneq V \otimes E \to \Mat_{2}(E)$ is an isomorphism and, thus, we may identify the quadric $Q_{V_E}$ with a product of two projective lines over $E$.
The nontrivial element of the Galois group of $E/F$ acts via $(x,y)\mapsto (y',x')$.
In particular, $Q_V(F)$ is identified with the set $\{(x,y)\in \PP^1(E)\times \PP^1(E)\ \vert \ x=y' \}$.
It follows that the base change $(X_V)_E$ of $X_V$ to $E$ (as a rigid analytic space over $E$) is isomorphic to the product of two upper half-planes over $E$:
\[
(X_V)_E \xlongrightarrow{\sim} \cH_E \times \cH_E.
\]
Moreover, the $\Gal(E/F)$-module $\cO((X_V)_E)^\times\hspace{-0.2em}/E^\times=M(\PP^1(E))_0 \oplus M(\PP^1(E))_0$ is an induced module.
In particular, the higher cohomology groups of $\Gal(E/F)$ with values in $\cO((X_V)_E)^\times\hspace{-0.2em}/E^\times$ vanish.
Thus, using Hilbert 90 one deduces
\[
\HH^1(\Gal(E/F),\cO((X_V)_E)^\times)=0,
\]
which implies that
\[
\cO(X_V)^\times\hspace{-0.2em}/F^\times\cong M(\PP^1(E),\Z)_0
\]
as $\SO(V)$-modules, where a matrix $A\in \GL_2(E)$ acts via Möbius transformations on the right hand side.
Moreover, since $\HH^1(X_V, \cO^\times)=\HH^1_{\mathrm{et}}(X_V, \cO^\times)$ by \cite[Theorem 8.2.3]{FvdP}, we may apply the Lyndon--Hochschild--Serre spectral sequence (see \cite[Proposition 2.6.12]{Huber}) to compute
\[
\HH^1(X_V,\cO^\times)=0.
\]
Finally, the divisor $\Delta_A$ attached to an element $A\in \cN_V\subseteq \PGL_2(E)$ matches the corresponding twisted diagonal divisor.
One readily checks that the spaces $\Div_{\spe}^\dagger(X_V)$ and $\cF^\dagger(\cN_V,\Z)$ are isomorphic.

\begin{Rem}\label{rem: flawless}
The calculations above show that in all cases $\cO(X_V)^\times/F^\times$ is the $\Z$-linear dual of a flawless $\Ort(V)$-representation in the sense of \cite[Definition 2.1]{Gehrmannautomorphic}.
In the split case, we may decompose $\cO(X_V)^\times/\times\cong M(\PP^1(F))_0 \oplus M(\PP^1(F))_0$.
On the first summand $G\times G$ acts via projection onto the first coordinate.
The standard identification of measures on $\PP^1(F)$ of total mass $0$ with harmonic cochains on the Bruhat--Tits tree (see, for example, \cite[Section 3.2.1]{plectic}) gives a short exact sequence of $G\times G$-modules of the form
\[
0 \too M(\PP^1(F))_0\too \Coind_{G_e\times G}^{G\times G}(\chi_{\orient}) \too \Coind_{G_v\times G}^{G\times G} (\Z)\too 0.
\]
Likewise, we get a resolution of the second summand.
In the nonsplit case, the action of $\SO(V)$ on $X_V$ can be extended to an action of $G_E\coloneq \PGL_2(E)$.
Writing $G_{E,v}, G_{E,e}\subseteq G_E$ for the stabilizer of a vertex respectively nonoriented edge in the Bruhat--Tits tree of $\PGL_2(E)$, we get the short exact sequence
\[
0 \too \cO(X_V)^\times/F^\times \too \Coind_{G_{E,e}}^{G_E}(\chi_{\orient}) \too \Coind_{G_{E,v}}^{G_E} (\Z)\too 0.
\]
of $G_E$-modules.
\end{Rem}


\section{Rigid meromorphic cocycles}
Fix a prime $p$ and a four-dimensional nondegenerate rational quadratic space $(V,q)$ such that 
\begin{itemize}
\item the real quadratic space $V_\infty\coloneq V\otimes \R$ has signature $(r,s)$ with $r \geq s$ and
\item $V_p\coloneq V\otimes \Q_p$ is not anisotropic.
\end{itemize}
Consider the $p$-adic symmetric space $X_p\coloneq X_{V_p}$ and abbreviate $\cA\coloneq \cO(X_p)$.
Let $\cM$ be the field of rigid meromorphic functions on $X_p$, that is, the fraction field of $\cA$.
Finally, let $\Div^\dagger(X_p)$ be the space of locally finite divisors on $X_p$.
Sending a nonzero rigid meromorphic function to its divisor defines an $\Ort(V_p)$-equivariant homomorphism
\[
\div\colon \cM^\times \too \Div^\dagger(X_p).
\]
The Picard group of $X_p$ vanishes by the results of Section \ref{sec: spaces} and, therefore, the sequence
\[
0\too \cA^\times \too \cM^\times \xlongrightarrow{\div} \Div^\dagger(X_p) \too 0
\]
of $\Ort(V_p)$-modules is exact.
Consider the set $\cP_V$ of positive definite lines in $V$, which we view as a subset of the local set $\cN_{V_p}$.
Given $\ell\in \cP_V$ we put $\Delta_{\ell,p}\coloneq \Delta_{\ell}\subseteq X_p$ to emphasize that this is a divisor on the symmetric space at $p$.
A locally finite special divisor $D\in \Div_{\spe}^\dagger(X_p)$ is called \emph{locally finite rational quadratic} if it supported on $\Delta_{\ell,p}$ with $\ell \in \cP_V$.
Write $\Div_{\ratq}^\dagger(X_p)$ for the space of locally finite rational quadratic divisors and $\cM^\times_{\ratq}$ for the preimage of $\Div_{\ratq}^\dagger(X_p)$ under the divisor map.
It follows that the sequence
\begin{align}\label{eq: divexact}
0\too \cA^\times \too \cM^\times_{\ratq} \xlongrightarrow{\div} \Div_{\ratq}^\dagger(X_p) \too 0
\end{align}
of $\Ort(V)$-modules is exact.
Note that the exactness of the sequence above was proven in \cite[Proposition 3.2]{DGL} for quadratic spaces of arbitrary dimension under the assumption that the quadratic space admits a self-dual $\Z_p$-lattice.
Given a subset $A\subseteq \cP_V$ we define
\[
\cF^\dagger_p(A,\Z)\coloneq \{f\in \cF^\dagger(\cN_{V_p},\Z) \mid \supp(f)\subseteq A\}.
\]
By the calculations of Section \ref{sec: spaces}, the $\Ort(V)$-module $\Div_{\ratq}^\dagger(X_p)$ is canonically isomorphic to the space
$\cF^\dagger_p(\cP_V,\Z)$.

For the remainder of the article, we fix a $p$-arithmetic subgroup $\Gamma\subseteq \SO(V)$.
In Section 4.2.4 of \cite{DGL}, it is explained how elements of $\HH^s(\Gamma, \cM^\times_{\ratq})$ can be evaluated a \emph{special points} of $X_p$.
Moreover, the special values for certain classes in $\HH^s(\Gamma, \cM^\times_{\ratq})$ are conjectured to be algebraic.
These classes, which are called rigid meromorphic cocycles in \textit{loc.cit.}, are characterized by their image in $\HH^s(\Gamma, \Div^\dagger_{\ratq}(X_p))$ under the divisor map.
In view of the long exact sequence of $\Gamma$-cohomology groups associated to the short exact sequence \eqref{eq: divexact}, one needs to get a handle on the cohomology groups $\HH^s(\Gamma, \Div^\dagger_{\ratq}(X_p))$, $\HH^{s}(\Gamma, \cA^\times)$ and $\HH^{s+1}(\Gamma, \cA^\times)$ in order to describe the space of all rigid meromorphic cocycles.
By the results of Section \ref{sec: spaces} the description of these spaces will depend very much on whether $V_p$ is the direct sum of two hyperbolic planes or not.
In the first case we say $V$ is \emph{split at $p$}, in the letter case we say $V$ is \emph{nonsplit at $p$}.

\subsection{Analytic (theta) cocycles}
Let us start by giving some qualitative statements about the cohomology with analytic coefficients.
In view of the short exact sequence
\[
0 \too \Q_p^\times \too \cA^\times \too \cA^\times/\Q_p^\times \too 0,
\]
one first has to consider the cohomology groups with trivial coefficients and with coefficients in $\cA^\times/\Q_p^\times$.
Since $p$-arithmetic groups are of type (VFL), it follows that the $\HH^i(\Gamma,\Z)$ is finitely generated and that the canonical homomorphisms
\[
\HH^i(\Gamma,\Z) \otimes \Q_p^\times \too \HH^i(\Gamma,\Q_p^\times)
\]
is injective and has finite cokernel for every $i \geq 0$.
Moreover, the base change map $\HH^i(\Gamma,\Z)\otimes \C \to \HH^i(\Gamma,\C)$ is an isomorphism and the latter group is computed in \cite{BFG}.
\begin{Pro}\label{pro: trivialcohomology}
The cohomology group $\HH^s(\Gamma,\Q_p^\times)$ is finite if $s=1$.
The cohomology group $\HH^{s+1}(\Gamma,\Q_p^\times)$ is finite if $V$ is split at $p$.
\end{Pro}
\begin{proof}
In case $\SO(V)$ is almost $\Q$-simple, this is a direct consequence of the main result of \cite{BFG} (see \cite[Proposition 4.5]{DGL} for more details).
If not, one can apply the Künneth theorem to reduce to the calculation of the cohomology of $p$-arithmetic subgroups of quaternion groups (cf.~\cite[Lemma 3.9]{Sprehe}). 
\end{proof}
Note that in case $V$ is nonsplit at $p$, the cohomology groups $\HH^{s+1}(\Gamma,\Q_p^\times)$ are in general nontorsion and related to automorphic representations of trivial cohomological weight, which are Steinberg at $p$.
By the computations of Section \ref{sec: spaces}, the $\SO(V_p)$-representation $\cA^\times\hspace{-0.2em}/\Q_p^\times$ is the $\Z$-dual of a flawless $\SO(V_p)$-representation in the sense of \cite[Definition 2.1]{Gehrmannautomorphic}.
It follows by Proposition 3.2 of \textit{loc.cit.~}that the cohomology groups $\HH^i(\Gamma, \cA^\times\hspace{-0.2em}/\Q_p^\times)$ are finitely generated $\Z$-modules.
Combining this result with Proposition \ref{pro: trivialcohomology} yields the following:
\begin{Cor}\label{cor: finiteness}
The  group $\HH^{s}(\Gamma,\cA^\times)$ is finitely generated if $s=1$.
The group $\HH^{s+1}(\Gamma,\cA^\times)$ is finitely generated if $V$ is split at $p$.
\end{Cor}
The group $\HH^{s+1}(\Gamma,\cA^\times)$ takes the role of the Picard group in the theory of rigid meromorphic cocycles.
Corollary \ref{cor: finiteness} should be read as the statement that the connected component of the Picard group is trivial if $V$ is split at $p$.
On the other hand, if $V$ is nonsplit at $p$, $\HH^{s+1}(\Gamma,\cA^\times)$ has a torus-like component that is related to automorphic forms that are Steinberg at $p$ (at least if $\Gamma$ is small enough). 
In other words, the Picard group has nontrivial connected component.

More refined results about the cohomology groups can be deduced from \cite{BFG} and the short exact sequences of Remark \ref{rem: flawless}.
As an example, let us consider the following case:
\begin{Pro}\label{pro: trivial}
Assume that $s=1$ and $V$ is split at $p$. 
Then $\HH^1(\Gamma, \cA^\times)$ and $\HH^1(\Gamma, \cA^\times\hspace{-0.2em}/\Q_p^\times)$ have rank $2$.
\end{Pro}
\begin{proof}
Since $V$ is split at $p$, we may choose an identification $V_p\cong\Mat_2(\Q_p)$, which induces an isomorphism $X_p$ with the product of two copies of the $p$-adic upper half-plane $\cH_p\coloneq\cH_{\Q_p}$ over $\Q_p$.
We view elements $\alpha\in \Gamma$ as tuples
\[
\alpha=(\alpha_1,\alpha_2)=\left(\begin{pmatrix} a_1 & b_1 \\ c_1 & d_1 \end{pmatrix} , \begin{pmatrix} a_2 & b_2 \\ c_2 & d_2 \end{pmatrix}\right)
\]
via the local identification of $\SO(V_p)$ as a subquotient of $\GL_2(\Q_p)\times \GL_2(\Q_p)$.
Recall the decomposition $\cA^\times\hspace{-0.2em}/\Q_p^\times \cong M(\PP^1(\Q_p),\Z)_0 \oplus M(\PP^1(\Q_p),\Z)_0$ from Theorem \ref{thm: functions}.
An element $\alpha$ acts on each direct sum via the according coordinate projection.
Let us first show that the rank of $\HH^1(\Gamma, M(\PP^1(\Q_p),\Z)_0 )$ is bounded by $1$, where we let $\Gamma$ acts via the first coordinate projection.
We may replace $\Gamma$ by a finite-index subgroup such that the action of $\Gamma$ on the Bruhat--Tits tree of $\PGL_2(\Q_p)$ via the first coordinate projection is orientation-preserving. 
From strong approximation one deduces that the set of vertices of $\cT$ decomposes into two $\Gamma$-orbits and the set of unoriented edges of $\cT$ consists of a single $\Gamma$-orbit.
Let $v,v^{\prime}$ be two adjacent edges of $\cT$ connected by the unoriented edge $e$.
Writing $\Gamma_v$, $\Gamma_{v^{\prime}}$ and $\Gamma_{e}$ for the stabilizers of $v$, $v^{\prime}$ and $e$ in $\Gamma$, the first short exact sequence of Remark \ref{rem: flawless} yields the short exact sequence
\[
0 \too M(\PP^1(\Q_p),\Z)_0 \too \Coind_{\Gamma_e}^{\Gamma}(\Z) \too \Coind_{\Gamma_v}^{\Gamma}(\Z) \oplus \Coind_{\Gamma_{v^{\prime}}}^{\Gamma}(\Z) \too 0
\]
of $\Gamma$-modules.
Using Shapiro's lemma, the associated long exact sequence yields the exact sequence
\[
\HH^0(\Gamma_e,\Z)\too \HH^0(\Gamma_v,\Z)\oplus \HH^0(\Gamma_{v^\prime},\Z) \too \HH^1(\Gamma, M(\PP^1(\Q_p),\Z)_0)\too \HH^1(\Gamma_e,\Z)
\]
in cohomology
Note that $\Gamma_e$ defines a lattice in the Lie group $\SO(V_\infty)\times \PGL_2(\Q_p)$ and, hence, $\HH^1(\Gamma_e,\Z)=0$ by Margulis' normal subgroup theorem (see \cite[Chapter VIII, Theorem 2.6]{Margulis}).
Note that the vanishing of $\HH^1(\Gamma_e,\Z)=0$ can also be deduced from \cite{BFG}.
A straightforward calculation shows that the cokernel of
\[
\HH^0(\Gamma_e,\Z)\too \HH^0(\Gamma_v,\Z)\oplus \HH^0(\Gamma_{v^\prime},\Z) 
\]
has rank $1$.
Since $\HH^1(\Gamma,\Q_p^\times)$ is finite, it remains to show that $\HH^1(\Gamma, \cA^\times/\Q_p^\times)$ has rank at most $2$.
It is easy to see that the two $\cA^\times$-valued $\Gamma$-cocycles 
\[
j_i(\alpha)((z_1,z_2))= \det(\alpha_i) (-c_i z_i + a_i),\quad i=1,2,
\]
generate linearly independent cohomology classes.
\end{proof}
In case $s=1$ and $V$ is nonsplit at $p$, the group $\HH^1(\Gamma, \cA^\times/\Q_p^\times)$ is related to automorphic forms that are Steinberg at $p$.
Via the exceptional isomorphism between $\SO(V)$ and a (quaternionic) Bianchi group, this cohomology group is identified with the one underlying Trifkovi\'c's extensions of the theory of Stark--Heegner points to the Bianchi setting (cf.~\cite{Trifkovic}).
Although there is no bound on the rank of $\HH^1(\Gamma, \cA^\times/\Q_p^\times)$ in this case, applying the same techniques as in \cite[Section 3.3]{Gehrmanninvertible} shows that $\HH^1(\Gamma, \cA^\times)$ has rank $1$.

\subsection{Divisor-valued cohomology groups}\label{sec: divcohomology}
The aim of this section is to calculate the cohomology of $\Gamma$ with values in the group of locally finite rational quadratic divisors on $X_p$ in degree less or equal to $s$ under additional assumptions on $V$.
Given a subset $A\subseteq \cP_V$ write $\Div^\dagger_{\ratq,A}(X_p)\subseteq \Div^\dagger_{\ratq}(X_p)$ for the subspace of divisors supported on $A$.
\subsubsection{The definite case.}
Let us first consider the case that $V$ is positive-definite.
\begin{Lem}\label{lem: definiteorbits}
Assume $V$ is positive-definite.
Let $\Orb\subseteq \cP_V$ be a $\Gamma$-orbit.
Then $\HH^0(\Gamma, \Div^\dagger_{\ratq,\Orb}(X_p))$ is free of rank $1$.
A generator of the cohomology group is given by 
\[
\sD_{\Orb}\coloneq \sum_{\ell \in \Orb} \Delta_{\ell,p}.
\]
\end{Lem}
\begin{proof}
One only has to check that $\sD_{\Orb}$ is locally finite.
Since $\Gamma$ is a $p$-arithmetic subgroup, there exists a $\Z[1/p]$-lattice $L \subseteq V$ and a positive rational number $m$ such that
\[
\Orb\subseteq \Orb(m,L)\coloneq \{\Span(v)\mid v\in L,\ q(v)=m\}.
\]
In view of the isomorphism of $\Div^\dagger_{\ratq,\Orb}(X_p)$ with $\cF^\dagger(\Orb,\Z)$, it is enough to show that $\Orb(m,L)\cap K$ is finite for every $K$ of a compact open covering of $\cN_{V_p}$.
By considering the covering associated to a $\Z_p$-lattice $\Lambda\subseteq V_p$ in Remark \ref{rem: isotropy},
it suffices to show that the sets
\[
\{ v\in L \mid q(v)=m,\ \iso_{\Lambda}(\Span(v))=r\}
\]
are finite for every $r\in \Z$.
One easily checks that the set $\{v \in V_p \mid q(v)=m,\ \iso_{\Lambda}(\Span(v))=r\}$ is contained in $p^t \Lambda$ for some $t\in \Z$.
Therefore, the claim follows from the finiteness of representation numbers for the definite $\Z$-lattice $L\cap p^t \Lambda$.
\end{proof}
Lemma \ref{lem: definiteorbits} gives an injective homomorphism
\begin{equation}\label{eq: definitedivisors}
\Z[\Gamma\backslash \cP_V] \intoo \HH^0(\Gamma, \Div^\dagger_{\ratq}(X_p)),\quad
\sum_{\Orb\in\Gamma\backslash \cP_V} n_\Orb\ \Orb \mapstoo \sum_{\Orb\in\Gamma\backslash \cP_V} n_\Orb\ \sD_\Orb 
\end{equation}
in case $V$ is positive-definite.

\begin{Pro}
Assume that $V$ is positive-definite and split at $p$.
Then \eqref{eq: definitedivisors} is an isomorphism.
\end{Pro}
\begin{proof}
Since $V$ is split at $p$, we may identify $X_p$ with $\cH_p\times \cH_p$.
This identification induces an action of $\Gamma$ on the product of the Bruhat--Tits tree $\cT$ of $\PGL_2(\Q_p)$ with itself.
The finiteness of class numbers of semi-simple algebraic groups implies that the set $\cV\times \cV$ of pairs of vertices of $\cT$ decomposes into a finite union of $\Gamma$-orbits.
Let $\{(v_1^i,v_2^i)\}$, $1\leq i \leq h$, be a set of representatives for these $\Gamma$-orbits.
Then for every $\Gamma$-orbit  $\Gamma$-orbit $\Orb\subseteq \cP_V$, the divisor $\sD_{\Orb}$ intersects the preimage of $(v_1^i,v_2^i)$ under the reduction map for at least one $1\leq i\leq h$.
Thus, a locally finite divisor can only be supported on finitely many $\Gamma$-orbits, which proves the assertion.
\end{proof}

\subsubsection{The indefinite case.}
Now assume that $s\geq 1$.
Contrary to the definite case, arithmetic subgroups of $\SO(V)$ are not finite in the indefinite case.
Let us first calculate the relevant cohomology groups of arithmetic subgroups of $\SO(V)$:
by \cite[Section 11.4]{BS}, every arithmetic subgroup of $\SO(V)$ is a virtual duality group of virtual cohomological dimension $rs-l_V$, where $0\leq l_V\leq 2$ denotes the Witt index of $V$.
The duality module can be described explicitly:
let $D_V$ be the Steinberg module of $\SO(V)$ (see for example \cite{Reeder}).
Then $D_V$ is the duality module for every torsion-free arithmetic subgroup of $\SO(V)$.
If $V$ is anisotropic, the module $D_V$ is equal to $\Z$ equipped with the trivial $\SO(V)$-action.
If $V$ is not anisotropic, $D_V$ is a free $\Z$-module of infinite rank.
Given an arbitrary subgroup $\Gammao\subseteq \SO(V)$ and a $\Gammao$-module $M$, we define the \emph{cohomology with compact supports} via
\[
\HH^i_c(\Gammao,M)\coloneq \HH^{i-l_V}(\Gammao, \Hom_{\Z}(D_V,M)).
\]
Note that $\HH^i_c(\Gammao,M)=\HH^i(\Gammao,M)$ if $V$ is anisotropic.
In case $\Gammao\subseteq \SO(V)$ is a torsion-free arithmetic subgroup, duality yields a canonical isomorphism
\begin{equation}\label{eq: duality}
\HH^i_c(\Gammao,M)\cong \HH_{rs-i}(\Gammao, M).
\end{equation}

Given an arithmetic group $\Gammao\subseteq \SO(V)$ and a $\Gammao$-stable subset $A\subseteq \cP_V$, we may decompose
\begin{align*}
\Z[A]=\bigoplus_{\Orb\in \Gammao\backslash A}\Z[\Orb].
\end{align*}
Since $\Z$ has a resolution by finitely generated projective $\Z[\Gammao]$-modules, it follows that the canonical map
\[
\bigoplus_{\Orb\in \Gammao\backslash A}\HH^\ast(\Gammao,\Z[\Orb])
\xlongrightarrow{\sim}\HH^\ast(\Gamma_\circ,\Z[A])
\]
is an isomorphism.
The same is true for cohomology with compact supports.
Consider the $\SO(V)$-stable subset
\[
\cP_V^c\coloneq \{\ell\in \cP_V \mid \ell^\perp\ \mbox{is anisotropic}\}
\]
of the set of positive-definite lines in $V$.
Note that in case the Witt index of $V$ is $2$, that is, $V$ is the direct sum of two hyperbolic planes, $\cP_V^c=\emptyset$.
More generally, for $A\subseteq \cP_V$ put $A^c\coloneq A \cap \cP_V^c$.
Let $\SO(V)^{+}$ be the subgroup of elements having trivial spinor norm.
\begin{Lem}\label{lem: arithmeticpreparation}
Let $\Gammao\subseteq \SO(V)$ be an arithmetic subgroup and $A\subseteq \cP_V$ a $\Gammao$-stable subset.
Then
\[
\HH^i_c(\Gammao, \Z[A^c])=\HH^i_c(\Gammao, \Z[A]) = 0 \quad \forall i < s.
\]
The inclusion $\Z[A^c]\into \Z[A]$ induces an isomorphism
\[
\HH^s_c(\Gammao, \Z[A^c])\xlongrightarrow{\sim}\HH^s_c(\Gammao, \Z[A]).
\]
If furthermore $\Gammao$ is a neat subgroup of $\SO(V_\infty)^{+}$, then
\[
\HH^s_c(\Gammao,\Z[\Orb])\cong \Z
\]
for every $\Gammao$-orbit $\Orb\subseteq \cP_V^c$.
\end{Lem}
\begin{proof}
The following is a slight variant of the proof of \cite[Proposition 3.2]{Gehrmannrational}.
First assume that $\Gammao$ is neat.
Let $\Orb\subseteq \cP_V$ be a $\Gammao$-orbit.
Fix an element $\ell\in \Orb$ and let $\Gamma_{\circ,\ell}\subseteq \Gammao$ be its stabilizer.
Then we may identify
\[
\Z[\Orb]\cong \Ind_{\Gamma_{\circ,\ell}}^{\Gammao}(\Z).
\]
Applying the duality isomorphism \eqref{eq: duality} and Shapiro's lemma for homology yields isomorphisms
\[
\HH^i_c(\Gammao,\Z[\Orb])\cong \HH_{rs-i}(\Gammao,\Z[\Orb])\cong \HH_{rs-i}(\Gamma_{\circ,\ell},\Z).
\]
Since $\Gamma_{\circ,\ell}$ is a neat arithmetic subgroup of $\SO(\ell^\perp)^{+}$, its cohomological dimension is less or equal to $(r-1)s$.
The cohomological dimension is equal to $(r-1)s$ if and only if $\ell \in \cP_V^c$, in which case $\HH_{(r-1)s}(\Gamma_{\circ,\ell},\Z)\cong \Z$.
This proves the neat case.
The general case follows from the fact that every arithmetic subgroup has a torsion-free normal subgroup of finite index and an easy application of the Hochschild--Serre spectral sequence.
\end{proof}
\begin{Rem}
The proof only requires that $\Gammao$ consists of elements, whose real spinor norm is trivial.
\end{Rem}

Let $A\subseteq \cP_V$ be a $\Gamma$-stable subset.
Consider the canonical homomorphism
\begin{align}\label{eq: divsum}
\bigoplus_{\Orb \in \Gamma\backslash A} \HH^{s}_c(\Gamma, \Div^\dagger_{\ratq,\Orb}(X_p)) \too \HH^{s}_c(\Gamma, \Div^\dagger_{\ratq,A}(X_p)) \tag{+}
\end{align}
induced by the inclusions $\Div^\dagger_{\ratq,\Orb}(X_p) \into \Div^\dagger_{\ratq,A}(X_p)$.

Suppose now that $V$ is split at $p$ and fix an isomorphism $V_p\cong \Mat_2(\Q_p)$.
This identification induces a map $\SO(V)^+\to \PSL_2(\Q_p)\times \PSL_2(\Q_p)$ and, thus, an action of $\SO(V)^+$ on the product of the Bruhat--Tits tree $\cT$ of $\PGL_2(\Q_p)$ with itself.
Suppose that $\Gamma$ is a subgroup of $\SO(V)^+$.
As in Section \ref{sec: SL2}, there is a $\Gamma$-equivariant decomposition
\begin{align*}
\Div^\dagger_{\ratq,A}(X_p)=\Div^\dagger_{\ratq,A}(X_p)^+ \oplus \Div^\dagger_{\ratq,A}(X_p)^-
\end{align*}
for every $\Gamma$-stable subset $A\subseteq \cP_V$.
Suppose that $\Gamma$ is a $p$-arithmetic subgroup of $\SO(V)^+$.
Let $v_0,v_1$ be two vertices of the Bruhat--Tits tree of $\PGL_2(\Q_p)$ connected by an edge $e$.
Furthermore, fix $g\in \PGL_2(\Q_p)$ with $ge=e$ and $g v_0 =v_1$.
Let $\Gamma_{i,j}$ be the stabilizer of $(v_i,v_j)$ and $\Gamma_e$ the stabilizer of $e$ in $\Gamma$, respectively.
By strong approximation, the image of $\Gamma$ in $\PSL_2(\Q_p)\times \PSL_2(\Q_p)$ is dense.
Thus, arguing as in Section \ref{sec: divisors} yields the short exact sequences
\begin{align*}
\begin{split}
0\to \Div^\dagger_{\ratq,A}(X_p)^+\to \bigoplus_{i=0}^{1} \Coind_{\Gamma_{i,i}}^{\Gamma}\left(\Z[A\cap K_i]\right) 
&\to \Coind_{\Gamma_e}^{\Gamma}\left(\Z[A\cap I]\right) \to 0\\
\intertext{and}
0\to \Div^\dagger_{\ratq,A}(X_p)^-\to \bigoplus_{i=0}^{1} \Coind_{\Gamma_{i,i-1}}^{\Gamma}\left(\Z[A\cap K_i g]\right) 
&\to \Coind_{\Gamma_e}^{\Gamma}\left(\Z[A\cap Ig]\right) \to 0
\end{split}
\end{align*}
of $\Gamma$-modules.

\begin{Thm}\label{thm: divisors}
Assume that $V$ is split at $p$. 
Let $\Gamma\subseteq \SO(V)$ be a $p$-arithmetic subgroup and $A\subseteq \cP_V$ a $\Gamma$-stable subset.
Then
\[
\HH^i_c(\Gamma, \Div^\dagger_{\ratq,A}(X_p)) = 0 \quad \forall i < s
\]
and the inclusion $\Div^\dagger_{\ratq,A^c}(X_p) \into \Div^\dagger_{\ratq,A}(X_p)$ induces an isomorphism
\[
\HH^s_c(\Gamma, \Div^\dagger_{\ratq,A^c}(X_p))\xlongrightarrow{\sim}\HH^s_c(\Gamma, \Div^\dagger_{\ratq,A}(X_p)).
\]
If $\Gamma$ is a neat subgroup of $\SO(V)^+$, then \eqref{eq: divsum} is an isomorphism and
\[
\HH^s_c(\Gamma, \Div^\dagger_{\ratq,\Orb}(X_p)) \cong \Z
\]
for every $\Gamma$-orbit $\Orb\subseteq A^c$.
\end{Thm}
\begin{proof}
Let $\Gamma\subseteq \SO(V)^+$ be a neat $p$-arithmetic subgroup.
The vanishing of the cohomology in degree $i < s$ follows from the long exact sequences in cohomology associated to the short exact sequences above for $A=\cP_V$ and $\cP_V^c$ together with Lemma \ref{lem: arithmeticpreparation}.
Furthermore, the cohomology of $\Gamma$ in degree $s$ with values in $\Div^\dagger_{\ratq,A}(X_p)^\pm$ and $\Div^\dagger_{\ratq,A^c}(X_p)^\pm$ agrees.
More precisely, on the plus part it is given by the kernel of the map
\[
\Z[\Gamma_{0,0}\backslash (A^c\cap K_0)] \oplus \Z[\Gamma_{1,1}\backslash (A^c\cap K_1)] \too \Z[\Gamma_e\backslash (A^c \cap I)].
\]
The projection of the kernel to either of the two direct summands is an isomorphism.
It remains to show that if $\ell,\ell' \in A \cap K_0$ lie in the same $\Gamma$-orbit, then they already lie in the same $\Gamma_0$-orbit.
This can be deduced from strong approximation applied to the group $\Spin(\ell^\perp)$.
The minus part can be treated equally.
The statement for general $\Gamma$ can be reduced to the neat case by applying the Hochschild--Serre spectral sequence.
\end{proof}
\begin{Rem}
Assume that $V$ is split at $p$ and $\Gamma$ is neat.
Then, Theorem \ref{thm: divisors} gives a complete description of $\HH^s(\Gamma, \Div^\dagger_{\ratq}(X_p))$ in case $V$ is anisotropic.
In general, the groups $\HH^s(\Gamma, \Div^\dagger_{\ratq}(X_p))$ and $\HH^s_c(\Gamma, \Div^\dagger_{\ratq}(X_p))$ behave rather differently.
For example, if $V$ is the orthogonal direct sum of two hyperbolic planes, the set $\cP_V^c$ is empty and, thus, $\HH^2_c(\Gamma, \Div^\dagger_{\ratq}(X_p))=0$.
But on the other hand, the computations of \cite[Section 5]{Sprehe} show that the group $\HH^2(\Gamma, \Div^\dagger_{\ratq}(X_p))$ has infinite rank.
\end{Rem}

\subsection{Intersection numbers}\label{sec: intersections}
Theorem \ref{thm: divisors} yields an abstract basis of the divisor-valued cohomology group in degree $s$ assuming that $V$ is anisotropic and split at $p$.
It remains to construct explicit cocycles representing the cohomology classes that generalize the construction in the definite case. 
The aim of this section is to construct an explicit class $\sD_\Orb$ of $\HH^s(\Gamma, \Div^\dagger_{\ratq,\Orb}(X_p))$ for every $\Gamma$-orbit $\Orb\subseteq \cP_V$.
Our construction is a slight generalization of the one made in Section 2.4.2 of \cite{DGL}, in which $V_p$ is assumed to admit a self-dual $\Z_p$-lattice.
Furthermore, we show that assuming that $V$ is anisotropic and split at $p$ the class $\sD_\Orb$ generates $\HH^s(\Gamma, \Div^\dagger_{\ratq,\Orb}(X_p))$.

\subsubsection{Comparison of fundamental classes.}
The main computation happens on the level of arithmetic subgroups.
As the same reasoning can be applied to other situations as well, e.g., the cycle-valued cohomology groups studied in \cite{Gehrmannrational}, we work in a more general framework:
fix a contractible smooth oriented manifold $X$ of dimension $m$ with an orientation-preserving action by a (discrete) group $\Group$ such that:
\begin{enumerate}[(i)]
\item\label{item: first} the action of $\Group$ on $X$ is properly discontinuous, that is, every $x\in X$ has a neighbourhood $U$ such that $\gamma U \cap U = \emptyset$ for all $\gamma\in \Group$, $\gamma \neq 1$, and
\item the quotient $\Group\backslash X$ is compact.
\end{enumerate}
It follows that $\Group$ is a Poincar\'e duality group of cohomological dimension $m$.
Furthermore, fix an integer $k\geq 1$ and a countable collection $\cC$ of oriented closed submanifolds of $X$ fulfilling the following assumptions:
\begin{enumerate}[(i)]
\setcounter{enumi}{2}
\item $\Delta$ is contractible, of codimension $k$ and its complement $X\setminus \Delta$ is homotopic to a $k-1$-sphere for all $\Delta \in \cC$,
\item $\gamma \Delta \in \cC$ for all $\gamma\in \Gamma$, $\Delta\in\cC$,
\item every $\Group$-orbit $\Orb\subseteq \cC$ is locally finite, that is, the set $\{\Delta\in \Orb \mid \Delta\cap K \neq \emptyset\}$ is finite for every compact subset $K\subseteq X$,
\item for every $\Delta\in\cC$ the quotient of $\Delta$ by its stabilizer $\GroupD$ in $\Group$ is compact, and
\item\label{item: Jaffee} for every $\Delta\in\cC$ there exists a finite index subgroup $\Group'\subseteq \Group$ such that the proper map
\[
\Group'_\Delta\backslash \Delta \intoo \Group'\backslash X
\]
is an embedding.
\end{enumerate}
It follows that $\GroupD$ is a Poincar\'e duality group of cohomological dimension $m-k$.
Moreover, the reduced homology group $\tilde{\HH}_{k-1}(X\setminus \Delta)$ is free of rank one.
The orientations on $X$ and $\Delta$ induce an isomorphism
\[
\tilde{\HH}_{k-1}(X\setminus \Delta) \xlongrightarrow{\sim} \Z.
\]
Since $\Group$ is of type (FL), taking $\Gamma$-cohomology commutes with direct sums.
In particular, the canonical map
\[
\bigoplus_{\Orb\in\Group\backslash \cC} \HH^i(\Group, \Z[\Orb])\xlongrightarrow{\sim}\HH^i(\Group, \Z[\cC])
\]
is an isomorphism for every $i$.
Moreover, arguing as in Lemma \ref{lem: arithmeticpreparation} yields isomorphisms
\begin{align}\label{eq: dualityShapiro}
\begin{split}
\HH^i(\Group, \Z[\Orb])&\xlongrightarrow{\sim}\HH_{m-i}(\Group,\Z[\Orb])\\
&\xlongrightarrow{\sim}\HH_{m-i}(\GroupD, \Z).
\end{split}
\end{align}
for every $\Group$-orbit $\Orb=\Group \Delta\subseteq \cC$.
It follows that
\[
\HH^i(\Group, \Z[\Orb])\cong\begin{cases}
0 & \mbox{for}\ i < k,\\
\Z & \mbox{for}\ i = k.
\end{cases}
\]
Our first aim is to construct an explicit generator of $\HH^k(\Group, \Z[\Orb])$:
consider the following $\Group$-stable subcomplex $\mathfrak{C}^{\cC}_\bullet(X)$ of the singular chain complex $C_\bullet(X)$:
\[
\mathfrak{C}^{\cC}_q(X)\coloneq
\begin{cases}C_{q}(X) & \mbox{for}\ q>k,\\
C_{q}(X\setminus \bigcup_{\Delta\in \cC}\Delta)& \mbox{for}\ q<k,
\end{cases}
\]
\[
\mathfrak{C}^{\cC}_{k}(X)\coloneq\left\{c\in C_{k}(X)\ \middle\vert\  d_{k}(c)\in \mathfrak{C}^{\cC}_{k-1}(X)\right\}.
\]
In the special case of special cycles on orthogonal symmetric spaces this subcomplex was considered in \cite[Section 2.4.1]{DGL} (see also \cite[Section 4.1]{Gehrmannrational}).
\begin{Lem}
The inclusion $\mathfrak{C}^{\cC}_\bullet(X)\into C_\bullet(X)$ is a quasi-isomorphism.
\end{Lem}
\begin{proof}
The claim is proven in the case case that $X$ is the symmetric space attached to an orthogonal group and $\cC$ is the collection of cycles coming from orthogonal complements of positive definite lines in \cite[Proposition 2.17]{DGL}.
The proof of \textit{loc.cit.~}works verbatim in our situation.
\end{proof}

For $\Delta\in \cC$ and a $c\in\mathfrak{C}^{\cC}_k(X)$ define their signed intersection number via
\[
\Delta \cdot c \coloneq [d_k(c)]\in \tilde{\HH}_{k-1}(X\setminus \Delta)=\Z.
\]
It is easy to check that the signed intersection number is $\Group$-equivariant, that is,
\[
\gamma.\Delta \cdot \gamma.c= \Delta \cdot c\quad \mbox{for all}\ \gamma\in\Group.
\]
Let $\cO\subseteq \cC$ be a finite union of $\Group$-orbits.
Since every $\Group$-orbit in $\cC$ is locally finite, the associated $\Group$-equivariant homomorphism
\[
\mathfrak{c}_\Orb^X\colon \mathfrak{C}^{\cC}_k(X) \too \Z[\Orb],\quad c \mapstoo \sum_{\Delta\in\Orb} (\Delta \cdot c)\ \Delta
\]
is well-defined.
By construction, we have $\mathfrak{c}_\Orb^X(d_{k+1}(c))=0$ for every $c\in \mathfrak{C}^{\cC}_{k+1}(X)$.
Therefore, as explained in \cite[Section 2.1]{DGL}, $\mathfrak{c}_\Orb$ gives rise to a cohomology class $\sD^{\Group}_\Orb \in \HH^k(\Group, \Z[\Orb])$.
For every finite-index subgroup $\Group'\subseteq \Group$ the assumptions \eqref{item: first} - \eqref{item: Jaffee} are still fulfilled and, therefore, one can also construct a class $\sD^{\Group'}_\Orb \in \HH^k(\Group', \Z[\Orb])$.
By construction, the restriction map
\[
\HH^k(\Group,\Z[\Orb])\too \HH^k(\Group',\Z[\Orb])
\]
sends $\sD^{\Group}_\Orb$ to $\sD^{\Group'}_\Orb$.

The main step in proving that $\sD^{\Group}_\Orb$ generates $\HH^k(\Group,\Z[\Orb])$ is to construct a particularly nice triangulation of $X$.
In particular, each simplex of the triangulation should be contained in the subcomplex $\mathfrak{C}^{\cC}_\bullet(X)$.
Crucial for the construction is the Eilenberg--Zilber triangulation of products of simplices:
given a finite ordered set $S$, write $\simplex^S$ for the simplex with set of vertices $S$.
For an integer $0\leq l \leq |S|$, denote by $\simplex^S_{\leq l}$ the face of $\simplex^S$ generated by the first $l$ vertices of $S$. Similarly, $\simplex^S_{\geq l}$ denotes the faces generated by the last $|S|-l+1$ simplices of $S$.
Write $\simplex^p$ for the standard $p$-simplex with vertices $0,\ldots,p$.
The vertices of the Eilenberg--Zilber triangulation of the product $\simplex^p \times \simplex^q$ are given by tuples $(i,j)$, $0\leq i \leq p$, $0\leq j \leq q$.
These pairs are partially ordered by the rule $(i,j)\prec (i',j')$ if and only if $i \leq i'$ and $j \leq j'$.
An $r$-simplex of $\simplex^p \times \simplex^q$ is a tuple $(v_0,\ldots,v_r)$ of $r+1$ vertices $v_0,\ldots,v_r$ such that $v_0 \leq v_1\leq\cdots \leq v_r$.
One easily verifies that the projection maps from $\pi_1\colon\Delta^p\times\Delta^q\to\Delta^p$ and $\pi_2\colon\Delta^p\times\Delta^q\to\Delta^q$ are simplicial.
\begin{Lem}\label{lem: simplex}
Let $b\in \simplex^q$ be an interior point.
\begin{enumerate}[(a)]
\item\label{lem: simplexa} No $(q-1)$-simplex of $\simplex^p\times\simplex^q$ intersects the subset $\simplex^p \times\{b\}$.
\item\label{lem: simplexb} There exists a unique $(p+q)$-simplex $\simplex$ of $\simplex^p\times\simplex^q$ such that $\simplex_{\leq q}$ intersects $\simplex^p \times\{b\}$.
The projection map $\pi_1\colon\simplex^p\times \simplex^q\to \simplex^p$ induces an isomorphism between $\simplex_{\geq q}$ and $\simplex^p$ that is order-preserving on vertices.
\end{enumerate}
\end{Lem}
\begin{proof}
Suppose that a simplex $\simplex$ of $\simplex^p\times\simplex^q$ intersects $\sigma^p\times\{b\}$.
Then $b\in \pi_2(\simplex)$.
Since $b$ is an interior point, this implies that $\pi_2(\simplex)=\Delta^q$ and, thus, that $\dim(\simplex)\geq q$.
A top-dimensional simplex $\simplex$ corresponds to a monotone path $v_0=(i_0,j_0),\ldots,v_{p+q}=(i_{p+q},j_{p+1})$ from $(0,0)$ to $(p,q)$.
Suppose that the $q$-face $\simplex|_{\leq q}$ intersects $\simplex^p \times\{b\}$.
Arguing as above, we see that $\pi_2$ induces an isomorphism between $\simplex|_{\leq q}$ and $\simplex^q$.
In other words, $\{j_i \mid 0\leq l \leq q\}=\{0,\ldots,q\}$.
Since the path $v_0,\ldots,v_{p+q}$ is monotone, we see that $v_l=(0,l)$ for $l\leq q$ and $v_l=(l-q,q)$ for $l\geq q$.
Clearly, the projection $\pi_1$ induces an order-preserving isomorphism between $\sigma_{\geq q}$ and $\sigma^p$.
\end{proof}

\begin{Lem}\label{lem: intersections}
Let $\Orb\subseteq \cC$ be a $\Group$-orbit.
The cohomology class $\sD^{\Group}_\Orb$ generates the free $\Z$-module $\HH^k(\Group, \Z[\Orb])$.
\end{Lem}
\begin{proof}
Fix a base point $\Delta\in \cC$ and a subgroup $\Group'$ as in assumption \eqref{item: Jaffee}.
The set $\Orb$ decomposes into finitely many $\Group'$-orbits $\Orb_1,\ldots,\Orb_h$ with $\Orb_1=\Group'.\Delta$.
The cohomology group and the class $\sD^{\Group'}_\Orb$ decompose accordingly:
\[
\HH^k(\Group', \Z[\Orb])=\bigoplus_{i=1}^{h}\HH^k(\Group', \Z[\Orb_i])\quad\mbox{and}\quad \sD^{\Group'}_\Orb=\sum_{i=1}^{h} \sD^{\Group'}_{\Orb_i}.
\]
One readily checks that it is enough to prove the statement for $\Group'$ and the orbit $\Orb_{1}$.
Thus, we may assume that $\Group=\Group'$.

Let us make the isomorphisms of \eqref{eq: dualityShapiro} more explicit:
write $M\coloneq \Group\backslash X$ and consider $N\coloneq \Group_\Delta\backslash \Delta$ as a submanifold of $M$.
Let $D_N$ be a closed tubular neighbourhood of $N$.
Locally on $N$ the tubular neighbourhood can be trivialized as $U\times D^k$, where $D^k\subseteq \R^k$ denotes the unit disc.
Choosing triangulizations of $D^k$ by a single $k$-simplex a triangulation $K(N)$ of $N$ that is fine enough, we get a triangulation of $D_N$ by glueing together the locally defined Eilenberg--Zilber product triangulations.
By construction, this is a triangulation relative to the boundary $\partial D_N$ of $N$.
Thus, we may extend it to a triangulation $K(M)$ of $M$.
Pulling back this triangulation via the covering map $X\onto M$ yields a $\Group$-stable triangulation $K(X)$ of $X$.
Similarly, we define $K(\Delta)$ as the pullback of $K(N)$ via the covering map $\Delta \to N$.
By Lemma \ref{lem: simplex}, \eqref{lem: simplexa} the $(q-1)$-simplices of $K(X)$ are disjoint from all $\Delta'\in \Orb$.
Let $\simplex_1,\ldots,\simplex_l$ be representatives for the $\Group$-orbits in the set of $m$-simplices of $\mathcal{K}$.
Note that there are finitely many orbits since the quotient $\Group\backslash X$ is compact.
We may assume that $\simplex_i \cap \Delta'=\emptyset$ for all $\Delta'\in \Orb$, $\Delta'\neq \Delta$.
For every $1\leq i \leq l$, choose orientation-preserving isomorphisms $c_i$ of the standard $d$-simplex with $\sigma_i$ and consider the chain $\sum_{i=1}^{l} c_i\in \mathfrak{C}^{\cC}_m(X)$.
Its image in $\HH_m(\Group\backslash X, \Z)\cong \HH_m(\Group, \Z)$ is a fundamental class.
Therefore, the image of $\sD^{\Group}_\Orb$ under the first isomorphism of \eqref{eq: dualityShapiro} is represented by the chain
\[
\sum_{i=1}^{l}\sum_{\Delta'\in \Orb} (\Delta' \cdot c_i|_{\leq k})\  c_i|_{\geq k}\otimes\Delta'  \in \mathfrak{C}^{\cC}_k(X) \otimes_\Z \Z[\Orb].
\]
The image under the second isomorphism of \eqref{eq: duality} is therefore represented by the chain
\begin{equation}\label{eq: fundclass}
\begin{split}
\sum_{i=1}^{l}\sum_{\Delta'\in \Orb} (\Delta' \cdot c_i|_{\leq k})\  c_i|_{\geq k} 
=\sum_{i=1}^{l}(\Delta \cdot c_i|_{\leq k})\  c_i|_{\geq k} \in \mathfrak{C}^{\cC}_{m-k}(X).
\end{split}
\end{equation}
All $c_i$ with $(\Delta \cdot c_i|_{\leq k})\neq 0$ are contained in a tubular neighbourhood $D_\Delta$ of $\Delta$.
Let $f\colon D_\Delta \to \Delta$ be the contraction map.
Lemma \ref{lem: simplex} \eqref{lem: simplexb} implies that $\{f_\ast(c_i|_{\geq k})\mid \Delta \cdot c_i|_{\leq k} \neq 0\}$ is a set of representatives for the $\Group_D$-orbits in the set of $(m-k)$-simplices of $K(\Delta)$.
Therefore, the image of \eqref{eq: fundclass} $\HH_{m-k}(\GroupD,\Z)$ is a fundamental class, which proves the assertion.
\end{proof}

\subsubsection{An explicit basis.}
Now let us go back to our specific situation and consider the space $X_\infty$ of maximal negative-definite subspaces of $V_\infty\coloneq V\otimes\R$.
It is a smooth contractible manifold of dimension $rs$, on which $\SO(V_\infty)$ acts orientation-preservingly.
Fix an orientation on $X_\infty$.
For every $\ell\in \cP_V$,
\[
\Delta_{\ell,\infty}=\{ Z\in X_\infty\mid Z\ \mbox{is orthogonal to}\ \ell\}
\]
defines a cycle on $X_\infty$ of codimension $s$.
By definition, we have
\begin{equation}\label{eq: cycles}
\gamma.\Delta_{\ell,\infty}=\Delta_{\gamma\ell,\infty}
\end{equation}
for all $\gamma\in\SO(V)$.
An orientation on $\ell$ induces an orientation on $\Delta_{\ell,\infty}$ (see \cite[p.~131]{KM}).
Assume that $\Gamma$ is a neat subgroup of $\SO(V)^+$.
Then, the stabilizer of $\ell \in \cP_V$ in $\Gamma$ acts trivially on the line $\ell$.
Hence, we may choose orientations on the cycles $\Delta_{\ell,\infty}$ such that \eqref{eq: cycles} is an equality of oriented cycles for all $\gamma\in \Gamma$.
Fix such a choice of orientations.
For any arithmetic subgroup $\Gammao\subseteq \Gamma$ the tuple $(X,\cC)=(X_\infty,\cP_V)$ satisfies the assumptions \eqref{item: first} - \eqref{item: Jaffee} except the compactness of the quotient spaces.
That assumption \eqref{item: Jaffee} is fulfilled follows from the Jaffee lemma (see \cite[Lemma 2.1]{KM}).
The quotient spaces are compact if and only if $V$ is anisotropic.
Let $\cO\subseteq \cP_V$ be a finite union of $\Gamma$ .
By arguing as in Section 2.4 of \cite{DGL}, one shows that the assignment
\[
\mathfrak{c}_\Orb\colon \mathfrak{C}^{\cP_V}_s(X_\infty) \too \Div^\dagger_{\ratq,\Orb}(X_p),\quad c \mapstoo \sum_{\ell\in\Orb} (\Delta_{\ell,\infty} \cdot c)\ \Delta_{\ell,p},
\]
is well-defined, $\Gamma$-equivariant and fulfils $\mathfrak{c}_\Orb(d_{s+1}(c))=0$ for every $c\in \mathfrak{C}^{\cP_V}_{s+1}(X_\infty)$.
Note that Lemma 2.10 of \textit{loc.cit.~}holds in our case since $\Div^\dagger_{\ratq,\Orb}(X_p)$ is isomorphic to the function space $\cF_p^\dagger(\Orb,\Z)$.
The cocycle $\mathfrak{c}_\Orb$ gives rise to a cohomology class $\sD_\Orb^{\Gamma}\in\HH^s(\Gamma, \Div^\dagger_{\ratq,\Orb}(X_p))$.

\begin{Pro}
Assume that $V$ is anisotropic and split at $p$ and that $\Gamma$ is a neat subgroup of $\SO(V)^+$.
Then $\sD_\Orb^{\Gamma}$ is a generator of the free abelian group $\HH^s(\Gamma, \Div^\dagger_{\ratq,\Orb}(X_p))$ for every $\Gamma$-orbit $\cO\subseteq \cP_V$.
\end{Pro}
\begin{proof}
Fix an isomorphism $V_p\cong \Mat_2(\Q_p)$.
Consider an open compact subset of $\cN_{V_p}$ of the form $K g$, where $K\subseteq \PGL_2(K_p)$ is the stabilizer of a vertex of the Bruhat--Tits tree of $\PGL_(\Q_p)$ and $g\in \PGL_2(\Q_p)$ is an arbitrary element.
Write $\Gammao\subseteq \Gamma$ for the subgroup of elements, whose image under the homomorphism $\Gammao \to \PGL_2(\Q_p)^2$ lie in $K\times K^g$.
By the proof of Theorem \ref{thm: divisors}, it is enough to show that the cohomology class associated with the assignment
\[
\mathfrak{C}^{\cC}_s(X_\infty) \too \Z[\Orb\cap K],\quad c \mapstoo \sum_{\ell\in\Orb \cap K g} (\Delta_{\ell,\infty} \cdot c)\ \ell,
\]
is a generator of $\HH^s(\Gammao,[\Orb\cap Kg])$.
But this is a direct consequence of Lemma \ref{lem: intersections}.
\end{proof}

\bibliographystyle{abbrv}
\bibliography{bibfile}

\end{document}